\documentclass[12pt]{article}
\makeatletter

\usepackage{etoolbox}

\newtoggle{kms@short}\togglefalse{kms@short}
\newtoggle{kms@nobib}\togglefalse{kms@nobib}

\usepackage{amssymb}
\usepackage{amsmath}
\usepackage{amsthm}
\usepackage{mathtools}
\usepackage{mathrsfs}
\usepackage{enumitem}
\usepackage{float}
\usepackage{setspace}
\usepackage{tikz}
\tikzset{every picture/.style={line width=0.75pt}}
\usepackage[labelfont=bf]{caption}
\allowdisplaybreaks

\iftoggle{kms@short}{%
  \theoremstyle{definition}%
  \newtheorem{definition}{Definition}%
  \theoremstyle{plain}%
}{%
  \newtheorem{theorem}{Theorem}%
  \newtheorem{lemma}{Lemma}%
  \newtheorem{corollary}{Corollary}%
  \theoremstyle{definition}%
  \theoremstyle{plain}%
}

\usepackage[colorlinks=true,allcolors=black]{hyperref}
\usepackage{url}

\iftoggle{kms@nobib}{}{%
  \usepackage[backend=biber,style=alphabetic,sorting=nyt,%
                  maxbibnames=99,giveninits=true]{biblatex}%
}

\providecommand{\kmstitlecaps}[1]{{\large\textbf{\MakeUppercase{#1}}}}

\newcommand{\kmsaddress}[3]{%
  #1:\ \textsc{#2}\\\nopagebreak
  \textit{Email address:} \texttt{#3}\par
}
\newcommand{\Addresses}[1]{{%
  \footnotesize
  #1%
}}
\makeatother

\makeatletter

\usepackage{amsmath}

\providecommand{\dd}{\mathrm{d}}

\makeatother

\makeatletter
\usepackage{xcolor}
\usepackage{hyperref}

\newif\ifcorrespondencedraft\correspondencedraftfalse

\def\eld@stop{\eld@stop}
\begingroup
  \catcode`\_=12\relax
  \gdef\eld@scan#1{%
    \ifx#1\eld@stop
    \else
      #1%
      \ifx#1.\allowbreak\fi
      \ifx#1_\allowbreak\fi
      \expandafter\eld@scan
    \fi}
\endgroup
\newcommand{\leandecl}[1]{%
  \begingroup\small\ttfamily
    \expandafter\eld@scan\detokenize{#1}\eld@stop
  \endgroup}

\usepackage{fvextra}

\usepackage{fontspec}
\newfontfamily\leanfont{JuliaMono}[
  Extension = .ttf,
  UprightFont = *-Regular,
  BoldFont = *-Bold,
  ItalicFont = *-RegularItalic,
  Scale = MatchLowercase,
]

\newcommand{\leaninline}[1]{{\leanfont #1}}
\definecolor{eldkeyi}{HTML}{0050A0}
\definecolor{eldkeyii}{HTML}{A33B00}
\definecolor{eldkeyiii}{HTML}{00664D}
\definecolor{eldkeyiv}{HTML}{7B2D8E}
\definecolor{eldkeyv}{HTML}{444444}
\definecolor{eldkeyvi}{HTML}{B10063}
\definecolor{eldkeyvii}{HTML}{775100}
\definecolor{eldkeyviii}{HTML}{00607A}

\newcommand{\leankey}[2]{%
  \ifcase#1\relax
    \textcolor{eldkeyi}{#2}%
  \or\textcolor{eldkeyi}{#2}%
  \or\textcolor{eldkeyii}{#2}%
  \or\textcolor{eldkeyiii}{#2}%
  \or\textcolor{eldkeyiv}{#2}%
  \or\textcolor{eldkeyv}{#2}%
  \or\textcolor{eldkeyvi}{#2}%
  \or\textcolor{eldkeyvii}{#2}%
  \or\textcolor{eldkeyviii}{#2}%
  \else\textcolor{eldkeyi}{#2}%
  \fi}

\DefineVerbatimEnvironment{LeanCodeKeyed}{Verbatim}{%
  fontsize=\footnotesize,
  formatcom=\leanfont,
  breaklines=true,
  breakanywhere=false,
  breaksymbolleft={},
  breaksymbolright={},
  xleftmargin=1em,
  commandchars=\£\«\»,
}

\usepackage{manyfoot}
\DeclareNewFootnote{Review}[fnsymbol]
\newcounter{reviewsymbol}
\newcommand{\reviewnotecolour}{blue}
\newcommand{\replycolour}{red}

\newcommand{\reviewnote}[2]{%
  \ifcorrespondencedraft
    \stepcounter{reviewsymbol}%
    \ifnum\value{reviewsymbol}>9\setcounter{reviewsymbol}{1}\fi
    \footnoteReview[\value{reviewsymbol}]{%
      \textcolor{\reviewnotecolour}{#1}\quad
      \textcolor{\replycolour}{#2}}%
  \fi}
\makeatother

\usepackage{xurl}
\usepackage{fvextra}
\usepackage{ragged2e}
\usetikzlibrary{arrows.meta, positioning, calc}

\newcommand{\leanversion}{v4.31.0-rc1}
\newcommand{\mathlibcommit}{542645a}
\newcommand{\projectcommit}{da8dd98}

\makeatletter
\let\eld@footnotetext\@footnotetext
\renewcommand{\@footnotetext}[1]{\eld@footnotetext{\RaggedRight #1}}
\makeatother

\makeatletter
\renewcommand\paragraph{\@startsection{paragraph}{4}{\z@}%
  {0.5\baselineskip \@plus 2pt \@minus 1pt}{-1em}{\normalfont\normalsize\bfseries}}
\makeatother
\AtBeginEnvironment{quote}{%
  \setlength{\topsep}{0.2\baselineskip}\setlength{\partopsep}{0pt}%
  \setlength{\parskip}{0pt}}

\DefineVerbatimEnvironment{LeanCode}{Verbatim}{%
  fontsize=\footnotesize,
  formatcom=\leanfont,
  breaklines=true,
  breakanywhere=true,
  breaksymbolleft={},
  breaksymbolright={},
  xleftmargin=1em,
}

\renewbibmacro{in:}{\ifentrytype{article}{}{\printtext{\bibstring{in}\intitlepunct}}}
\AtEveryBibitem{\iffieldequalstr{eprinttype}{arXiv}{\clearfield{doi}}{}}
\DeclareFieldFormat{doi}{doi\addcolon\space\href{https://doi.org/#1}{\mbox{\nolinkurl{#1}}}}

\begin{document}

\title{\kmstitlecaps{Formalising Linear Elliptic PDE Theory in Lean~4}}
\author{\textsc{Alejandro Jos\'e Soto Franco} \& \textsc{Kobe Marshall-Stevens}}
\date{\vspace{-4.5ex}}
\maketitle

\begin{abstract}
    \noindent We formalise in Lean~4, on top of Mathlib, the solvability of the
    Dirichlet problem for second-order linear elliptic operators in divergence
    form. The machine-verified results, with no \texttt{sorry} in the
    development, include the Poincar\'e inequality, the existence of weak
    solutions by the Lax--Milgram theorem, Rellich--Kondrachov compactness, the
    Fredholm alternative, the spectral theorem, interior regularity estimates, and
    the Sobolev embedding theorem. From these results we obtain a formalisation of
    classical solvability for sufficiently regular coefficients and data. Our
    Lean library includes a self-contained theory of Sobolev spaces developed
    independently of existing formalisations. Throughout the paper we associate
    each prose statement with the named machine-checked Lean declaration that
    discharges it. 
\end{abstract}

\begin{center}
\end{center}

{\setlength{\parskip}{0pt}\tableofcontents}

\section{Introduction}\label{sec: introduction}

Mathematical analysis has recently come within reach of interactive proof
assistants, but the
formalisation of partial differential equations (PDE) theory is still at an early stage. In \cite{armstrong-2026-for-de-gio}, interior De
Giorgi--Nash--Moser regularity theory was formalised in Lean~4 (see
\cite{demoura-2021-lea-4-the}) on top of Mathlib (see
\cite{themathlibcommun-2020-lea-mat-lib}), building in the process a library for Sobolev
spaces on bounded domains. The surrounding analytic infrastructure has grown alongside
it. A non-exhaustive list of examples include formalisations of the
Gagliardo--Nirenberg--Sobolev inequality in \cite{vandoorn-2024-int-wit-int}, Schwartz
functions, tempered distributions, and Sobolev spaces on
Euclidean space through the Fourier transform in \cite{doll-2025-for-sch-fun},
as well as the $h$-principle and sphere eversion in \cite{massot-2022-for-hpr-sph}. We
also mention that, beyond Lean, a Coq proof of the Lax--Milgram theorem was formalised
in
\cite{boldo-2017-coq-for-pro}. These developments suggest that a
substantial part of modern analysis can now be placed on machine-verified footing. 

To this end, the present work concerns formalisation, in Lean~4 on top of Mathlib, of
the solvability of the Dirichlet problem for general second-order linear elliptic
operators in divergence form. This theory is
central to modern PDE, with its treatment via
functional-analytic methods by now standard (e.g.~see \cite{
gilbarg-2001-ell-par-dif, evans-2010-par-dif-equ, brezis-2011-fun-ana-sob}). While we
focus on solvability for the Dirichlet problem specifically, we expect that the various
tools developed here bring much of the remaining standard theory for linear elliptic PDE (boundary regularity, variational methods, the Neumann problem, etc.)~well within reach
of formalisation within Lean.

In certifying a formal proof, a reader must confirm that the Lean code provided indeed expresses the intended statement. We therefore accompany our formalisation with a transcription from Lean to mathematical prose; each of the key results is stated in standard mathematical prose and accompanied by the Lean statement of the declaration that proves
it. 

\subsection{Mathematical preliminaries}\label{subsec: preliminaries}

Let $\Omega \subseteq \mathbb{R}^n$ be a bounded set. We consider second-order
linear partial differential operators $L$ in divergence form, with coefficients $a^{ij}, b^i, c \in L^\infty(\Omega)$ for each integer $1 \leq i,j
\leq n$, acting on
functions $u \in C^2(\Omega)$ by
\begin{equation}\label{eqn: divergence form operator}
    L u = -\sum_{j=1}^n D_j \Bigl( \sum_{i=1}^n a^{ij}\, D_i u \Bigr)
        + \sum_{i=1}^n b^i\, D_i u + c\, u.
\end{equation}
We assume that $L$ is uniformly elliptic in the sense that the matrix
of coefficients $a^{ij}$ is uniformly positive definite and bounded. Equivalently, we assume
that there exist constants $0 < \lambda \leq \Lambda < \infty$ such that for
each $\xi \in \mathbb{R}^n$
\begin{equation}\label{eqn: ellipticity}
    0 < \lambda |\xi|^2 \leq \sum_{i,j = 1}^na^{ij}(x) \xi^i\xi^j \leq \Lambda |\xi|^2
\end{equation}
for almost every $x \in \Omega$. Given $f \in L^2(\Omega)$, we address the solvability, both weakly and classically, of
the Dirichlet problem
\begin{equation}\label{eqn: dirichlet problem}
    \begin{cases}
        Lu  = f & \text{in } \Omega,\\
        u = 0 & \text{on } \partial \Omega.
    \end{cases}
\end{equation}
Multiplying the equation $L u = f$ by a test function and integrating by
parts yields a weak formulation of the Dirichlet problem on the Sobolev space
$H_0^1(\Omega)$. Namely, we say that a function $u \in H_0^1(\Omega)$ is a
weak solution of
    \eqref{eqn: dirichlet problem} if
    \begin{equation}\label{eqn: weak form}
        B[u, v] = \langle f, v \rangle_{L^2(\Omega)}
        \qquad\text{for all } v \in H_0^1(\Omega),
    \end{equation}
    where the bilinear form
    $B \colon H_0^1(\Omega) \times H_0^1(\Omega) \to \mathbb{R}$ associated with
    $L$ (e.g.~see \cite[\S 6.1.2]{evans-2010-par-dif-equ}) is defined by
    \begin{equation}\label{eqn: bilinear form}
        B[u, v]
        = \sum_{i,j = 1}^n\int_\Omega \bigl( a^{ij}\, D_i u \, D_j v + b^i (D_i u)\, v
            + c\, u v \bigr).
    \end{equation}
    If in addition $a^{ij} \in C^1(\Omega)$ and $u \in C^2(\Omega)$, then by the
    fundamental lemma of the calculus of variations a weak solution of
    (\ref{eqn: dirichlet problem}) satisfies $Lu = f$ almost everywhere in
    $\Omega$ and is in this sense a classical solution of
    (\ref{eqn: dirichlet problem}), whose boundary condition is expressed by the
    membership $u \in H_0^1(\Omega)$.

    A common approach to the classical solvability of
    (\ref{eqn: dirichlet problem}) proceeds in two stages. One first obtains a
    weak solution by functional-analytic methods, namely the Lax--Milgram theorem
    and the Fredholm alternative, and then shows by interior regularity
    estimates and the Sobolev embedding theorem that this weak solution is a
    classical solution whenever the coefficients and the data are sufficiently
    regular. We formalise this approach, largely following the presentation of
    \cite[Chapter~6]{evans-2010-par-dif-equ}, adapting several of the proofs
    from the notes of \cite[Chapters~1--3]{guo-2026-par-dif-equ} (based on a course on elliptic PDE taught by the second-named author).

\subsection{Formalisation contributions}\label{subsec: this work}
We provide a Lean~4 library, built on top of Mathlib, containing the formalisation of
statements leading to the classical solvability of (\ref{eqn: dirichlet problem}). The
Lean~4 library that accompanies this paper is referred to as the \textbf{library}
throughout the manuscript. This is a standalone Mathlib-based development
containing each declaration against which the
statements of Section~\ref{sec: statements} are discharged. The structure of the
library is further described in
Section~\ref{sec: library} and pinned at commit \texttt{\projectcommit} by our
code availability statement. The library can be found here:
\begin{center}
\begin{sloppy}\url{https://github.com/alejandro-soto-franco/EllipticPDE}\end{sloppy}
\end{center}

The main contributions contained in the library are the following:
\begin{enumerate}

\item\label{contrib: elliptic} A self-contained weak-derivative Sobolev layer, developed independently of the one in \cite{armstrong-2026-for-de-gio}, together with a formalised account built on top of it of the weak solvability of (\ref{eqn: dirichlet problem}). The precise statements formalised are detailed in Subsection \ref{subsec: solvability statements}.

\item\label{contrib: higher} Higher interior regularity for weak solutions of (\ref{eqn: dirichlet problem}) and the Sobolev embedding, leading to classical solvability. The precise statements formalised are detailed in Subsection \ref{subsec: regularity}.

\item\label{contrib: methodology} An explicit correspondence that associates with each prose statement the named machine-checked Lean declaration proving it. Section~\ref{sec: statements} transcribes this correspondence by printing each Lean statement between the result that it proves and the sketch of its proof. The definitions used in these statements are collected in Appendix~\ref{sec: appendix transcription}.
\end{enumerate}

\subsection{Formalisation framework}\label{subsec: framework}

Each step of a proof in Lean gives rise to an obligation to which we assign one
of four statuses. An obligation is \emph{discharged} if a
machine-checked Lean proof term establishes it, so that its correctness depends
on nothing beyond the kernel and the stated axioms. It is \emph{warranted} if it
is deferred to a located external source, \emph{routine} if it relies on
shared mathematical competence, and \emph{open} if no justification for it has
yet been given. The discharge of an obligation thus has two components, namely
the machine check of the type of the declaration and a human audit confirming
that the prose statement agrees with that type. A proof is called
\emph{formal} if every one of its steps is discharged, so that no obligation
remains.

The library was written with the language models Claude Opus 4.5 and Claude
Opus 5, used as agents in the Claude Code harness under the direction of the
first-named author. The authors chose the target theorems, the weak
formulation, the encoding of $H_0^1(\Omega)$ as a graph, the coefficient bundle, the route of each proof, and compared each Lean type with the corresponding
prose statement. The agents searched Mathlib and the library
for existing declarations and their exact types, decomposed each target into
named auxiliary declarations, attempted the tactic proofs, and proposed
improvements. No output of a model forms part of the trusted proof object, since
the Lean kernel checks every proof term by the same rules irrespective of how
the term was produced.

The criteria for acceptance were mechanical and can be reproduced from the
pinned commit. A proposed change was admitted to the library only if the whole
development built from a clean clone with no \texttt{sorry} and no warning, the
environment linter passed, and \texttt{\#print axioms} reported for every named
declaration exactly the three standard axioms of the Lean core library (namely
propositional extensionality (\texttt{propext}), the axiom of choice
(\texttt{Classical.choice}), and quotient soundness (\texttt{Quot.sound})). The
library pins those reports at every named declaration, so a dependence on a
further axiom fails the build.

\subsection{Structure}

The paper is organised as follows. In Section~\ref{sec: statements} we state the
main formalised results. Each result is followed by the Lean statement of the
declaration that proves it along with a sketch of proof naming the declarations on which the formal proof depends. Section~\ref{sec: library} describes the structure of the Lean library. Section~\ref{sec: discussion} concludes the paper with the discharge ratio, the part of the audit that the kernel cannot perform, the relation to prior work and
directions for further research. Appendix~\ref{sec: appendix transcription}
gives the Lean code for each definition used in the statements.

\section*{Declarations}

\noindent\textbf{Code availability.}\ \,The Lean library is openly
available at
\begin{sloppy}\url{https://github.com/alejandro-soto-franco/EllipticPDE}\end{sloppy}.
An archival release with a persistent identifier will be deposited upon
publication. The version described here is commit \texttt{\projectcommit},
built with Lean~4~\leanversion{} and Mathlib commit
\texttt{\mathlibcommit}.\par\smallskip
\noindent\textbf{Author contributions.}\ \,Both authors contributed to the
mathematical design and to the writing of the manuscript. The first-named author
carried out the Lean formalisation and its transcription into mathematical
prose.\par\smallskip
\noindent\textbf{Use of language-model assistance.}\ \,The Lean library and
this manuscript were written with the assistance of language-model agents under
the direction of the first-named author. No output of a language model forms part
of the trusted proof object. The models used, the division of work and the
acceptance checks are described in Subsection~\ref{subsec: framework}.\par\smallskip
\noindent\textbf{Acknowledgements.}\ \,The authors are grateful to James Guo
for the typeset notes \cite{guo-2026-par-dif-equ} based on courses taught by the
second-named author.\par

\section{Formalised statements}\label{sec: statements}

In this section we state the formalised results that lead to the classical
solvability of \eqref{eqn: dirichlet problem}. Subsection~\ref{subsec:
solvability statements} gives the statements that lead to weak solvability, and
Subsection~\ref{subsec: regularity} those that lead to the regularity of weak
solutions. We state only those formalised results leading directly to
the classical solvability of \eqref{eqn: dirichlet problem}, and so do not survey every formalised result in the library.

Each result is stated in standard mathematical language together with a
citation to a standard treatment that it formalises. The Lean
statement of the declaration that proves it follows the result and is itself
followed by a sketch of the proof. The proof sketches follow the cited proof and say
where the formal proof departs from it. This happens when a hypothesis is replaced
by one that is more convenient to state in Lean, when a constant that the cited
proof leaves implicit has to be produced, or when an argument is reorganised as
an induction that the proof assistant can check. The sketch also places the
result in the formal development by naming in footnotes the auxiliary
declarations on which its proof depends, so that the adaptations of the cited
proofs can be traced to the library.

\subsection{Weak solvability of the Dirichlet problem}\label{subsec: existence
statements}\label{subsec: solvability statements}\label{subsec: variational}

We begin with the Poincar\'e inequality on bounded sets:

\begin{theorem}[{Poincar\'e inequality, \cite[\S 5.6.1, Theorem~3]{evans-2010-par-dif-equ}}]\label{thm: poincare}
    Let $\Omega \subseteq \mathbb{R}^n$ be a bounded set. Then there is a
    constant $C_P \geq 0$, depending only on $n$ and $\Omega$, such that
    \begin{equation*}
        \| u \|_{L^2(\Omega)} \leq C_P \, \| \nabla u \|_{L^2(\Omega)}
        \qquad\text{for all } u \in H_0^1(\Omega) .
    \end{equation*}
\end{theorem}

\paragraph{\leandecl{poincare_H01_of_bounded}}
The Lean theorem is stated as
\begin{quote}
\begin{LeanCode}
theorem poincare_H01_of_bounded {Ω : Set (EuclideanSpace ℝ (Fin (n + 1)))}
    (hΩb : Bornology.IsBounded Ω) :
    ∃ C : ℝ, 0 ≤ C ∧ ∀ (U : H1amb Ω), U ∈ H01 Ω →
      ‖U 0‖ ^ 2 ≤ C * ∑ i : Fin (n + 1), ‖U i.succ‖ ^ 2
\end{LeanCode}
\end{quote}

\begin{proof}[Proof sketch]
    Whereas \cite[\S 5.6.1, proof of Theorem~3]{evans-2010-par-dif-equ} deduces the inequality from the
    Gagliardo--Nirenberg--Sobolev inequality, which requires $p < n$, the formal
    proof reduces it to one dimension on a box and obtains an explicit constant
    in every dimension (as in the statement of 
    \cite[Corollary~9.19]{brezis-2011-fun-ana-sob}). Let $Q = \prod_i (a_i, b_i)$ be such a box. On each line parallel to
    the $i$th axis a test function on $Q$ is compactly supported in
    $(a_i, b_i)$, so that by the one-dimensional inequality of
    Lemma~\ref{lem: poincare 1d} (proved in
    Section~\ref{sec: library} together with the Cauchy--Schwarz inequality of
    Lemma~\ref{lem: cauchy schwarz} on which it rests), the integral of its
    square along that line is bounded by $\tfrac{(b_i - a_i)^2}{2}$ multiplied by the integral of $(\partial_i u)^2$ along it. Integrating over the remaining directions by Fubini's
    theorem, we obtain the corresponding bound on $Q$ for each direction.
    Averaging these $n$ bounds with $L = \max_i(b_i - a_i)$ gives the
    inequality on $C_c^\infty(Q)$ with $C_P = L/\sqrt{2n}$, from which the
    inequality on $H_0^1(Q)$ follows by approximation. To pass from a box to a
    bounded set $\Omega$ we enclose it in a box $Q$ as above. A test function on
    $\Omega$ is then a test function on $Q$ whose integrals over the two
    sets coincide. Hence the
    constant for $Q$ is also admissible for $\Omega$. The formalised
    statement is in the squared form
    $\|u\|_{L^2(\Omega)}^2 \leq C \|\nabla u\|_{L^2(\Omega)}^2$ with $C = C_P^2$ and asserts
    only the existence of $C$, whereas the statement for boxes gives the
    admissible value $C = L^2/(2n)$. Every constant quoted below is the
    unsquared $C_P$.\footnote{\leandecl{EllipticPdes.Poincare.poincare_H01_of_bounded},
    with the statement for boxes \leandecl{poincare_H01_euclBox}.}
\end{proof}

We state the Lax--Milgram theorem next:

\begin{theorem}[{Lax--Milgram theorem, \cite[\S 6.2.1,
Theorem~1]{evans-2010-par-dif-equ}}]\label{thm: lax milgram}
    Let $(H, \langle \cdot, \cdot \rangle)$ be a Hilbert space,
    $B \colon H \times H \to \mathbb{R}$ be a bounded coercive bilinear form,
    with $B[u, u] \geq \beta \| u \|_H^2$ for some $\beta > 0$, and 
    $f \in H^{\ast}$. Then, there is a unique $u \in H$ such that
    \begin{equation*}
        B[u, v] = \langle f, v \rangle \qquad \text{for all } v \in H .
    \end{equation*}
\end{theorem}

\paragraph{\leandecl{lax_milgram}}
The Lean theorem is stated as
\begin{quote}
\begin{LeanCode}
theorem lax_milgram {H : Type*} [NormedAddCommGroup H] [InnerProductSpace ℝ H]
    [CompleteSpace H] {B : H →L[ℝ] H →L[ℝ] ℝ} (hB : IsCoercive B) (f : H →L[ℝ] ℝ) :
    ∃! u : H, ∀ v : H, B u v = f v
\end{LeanCode}
\end{quote}

\begin{proof}[Proof sketch]
    Since $B$ is bounded, $B[u, \cdot\,]$ is a continuous linear functional for each $u \in H$.
    The Riesz representation theorem therefore yields a bounded linear operator
    $A \colon H \to H$ with $\langle A u, v \rangle = B[u, v]$ for every $v \in H$.
    By coercivity, $\beta \| u \|_H^2 \leq B[u, u] = \langle Au, u \rangle
    \leq \| A u \|_H \| u \|_H$, so $A$ is injective with closed range.
    Applying the same bound on the orthogonal complement of the range, we deduce
    that $A$ is surjective. If $g$ denotes the Riesz representative of $f$, the
    solution is $u = A^{-1} g$. A second solution $u'$ satisfies
    $\langle A u', w \rangle = \langle g, w \rangle$ for every $w$, whence
    $A u' = A u$ and $u' = u$. No hypothesis is imposed on $B$ apart from
    boundedness and coercivity, which the Lean statement expresses respectively
    by the continuity built into the type of $B$ and by Mathlib's predicate
    \texttt{IsCoercive}, as elsewhere in the library. The solution satisfies
    $\| u \|_H \leq \beta^{-1} \| f \|_{H^{\ast}}$, as follows by combining
    coercivity with $B[u, u] = \langle f, u \rangle \leq \| f \|_{H^{\ast}}
    \| u \|_H$. Each a priori bound below is this estimate applied with the
    coercivity constant of the form in
    question.\footnote{\leandecl{EllipticPdes.lax_milgram}, with the estimate
    \leandecl{EllipticPdes.norm_weak_solution_le}.}
\end{proof}

We first apply Theorem~\ref{thm: lax milgram} to the Dirichlet problem for the
Laplacian.

\begin{corollary}[{Weak solvability of the Poisson equation, \cite[\S 6.2.2]{evans-2010-par-dif-equ}}]\label{cor: dirichlet weak solution}
    Let $\Omega \subseteq \mathbb{R}^n$ and let $C \geq 0$ satisfy the
    test-function bound $\| \varphi \|_{L^2(\Omega)}^2 \leq C \sum_{i = 1}^n
    \| \partial_i \varphi \|_{L^2(\Omega)}^2$ for every $\varphi \in C^\infty_c(\Omega)$. Then, for every $f \in H^{-1}(\Omega)$ there is a unique
    $u \in H_0^1(\Omega)$ with $ \sum_{i = 1}^n \langle \partial_i u, \partial_i v \rangle_{L^2(\Omega)} = f(v)$ for every $v \in H_0^1(\Omega)$.
\end{corollary}

\paragraph{\leandecl{poisson_weak_solution}}
The Lean theorem is stated as
\begin{quote}
\begin{LeanCode}
theorem poisson_weak_solution (Ω : Set (EuclideanSpace ℝ (Fin d))) (C : ℝ) (hC : 0 ≤ C)
    (hbase : ∀ {φ : EuclideanSpace ℝ (Fin d) → ℝ} (h : IsTestFn Ω φ),
      ‖(h.testGraph 0 : L2D Ω)‖ ^ 2 ≤ C * ∑ i : Fin d, ‖h.testGraph i.succ‖ ^ 2)
    (f : H01 Ω →L[ℝ] ℝ) :
    ∃! u : H01 Ω, ∀ v : H01 Ω, laplaceBilin Ω u v = f v
\end{LeanCode}
\end{quote}

\begin{proof}[Proof sketch]
    We apply Theorem~\ref{thm: lax milgram} to $B[u,v] =  \sum_{i = 1}^n \langle \partial_i u, \partial_i v \rangle_{L^2(\Omega)}$ with $H = H_0^1(\Omega)$. Since
    each of the $n$ terms of $B$ is bounded by the product of the norms
    by the Cauchy--Schwarz inequality, $|B[u, v]| \leq n \| u \|_{H_0^1(\Omega)} \| v \|_{H_0^1(\Omega)}$. Both $B[u, u] = \sum_{i = 1}^n \| \partial_i u \|^2_{L^2(\Omega)}$ and the assumed test-function bound, extended to
    $H_0^1(\Omega)$ by density, imply that
    $B[u, u] \geq (C + 1)^{-1} \| u \|_{H_0^1(\Omega)}^2$. We then conclude by applying Theorem~\ref{thm: lax milgram} with $\beta = (C + 1)^{-1}$. The formal statement takes the test-function bound as a hypothesis and imposes no condition on $\Omega$, so that it also applies to unbounded sets on which such a bound holds. The hypothesis enters the proof only through the coercivity estimate.\footnote{\leandecl{EllipticPdes.poisson_weak_solution}, with
    coercivity proved in \leandecl{laplaceBilin_coercive}.}
\end{proof}

To obtain coercivity for more general operators of the form \eqref{eqn: divergence form operator} we have:

\begin{theorem}[{G{\aa}rding inequality, \cite[\S 6.2.2,
Theorem~2]{evans-2010-par-dif-equ}}]\label{thm: garding}
    Let $\Omega \subseteq \mathbb{R}^n$. Let $L$ be the operator
    \eqref{eqn: divergence form operator} with ellipticity constant $\lambda$ as
    in \eqref{eqn: ellipticity}. If
    $\gamma = \lambda/2 + \|c\|_{L^\infty(\Omega)} + n\,(\max_i \|b^i\|_{L^\infty(\Omega)})^2/(2\lambda)$, then
    \begin{equation*}
        \tfrac{\lambda}{2}\|u\|_{H_0^1(\Omega)}^2 \leq B[u,u] + \gamma\|u\|_{L^2(\Omega)}^2
        \qquad \text{for all } u \in H_0^1(\Omega) .
    \end{equation*}
\end{theorem}

\paragraph{\leandecl{garding}}
The Lean theorem is stated as
\begin{quote}
\begin{LeanCode}
theorem garding (Ω : Set (EuclideanSpace ℝ (Fin d))) (U : H01 Ω) :
    Op.lam / 2 * ‖U‖ ^ 2
      ≤ Op.fullBilin Ω U U + Op.gardingγ * ‖(U : H1amb Ω) 0‖ ^ 2
\end{LeanCode}
\end{quote}

\begin{proof}[Proof sketch]
    
    One first shows that $B$ is bounded, which follows by the pointwise bounds on the coefficients\footnote{\leandecl{EllipticPdes.Sobolev.EllipticCoeff.bilin}
    whose bound is part of its construction.}. By ellipticity, for the principal part we see that $B_A[u,u] = \sum_{i,j} \langle a^{ij}\,\partial_i u, \partial_j u\rangle \geq \lambda \|\nabla u\|_{L^2(\Omega)}^2$\footnote{\leandecl{EllipticPdes.Sobolev.EllipticCoeff.bilin_coercive}.}. By the Peter--Paul inequality, the first order terms of $B$ are absorbed, and the
    zeroth-order terms are controlled pointwise by $\|c\|_{L^\infty(\Omega)}$. The desired inequality follows by using these facts on the difference between $B_A$ and $B$. Whereas the cited statement leaves $\gamma$ unspecified, the formal statement gives its value.\footnote{\leandecl{garding}
    which admits any almost-everywhere bounds on $b$ and $c$ in place of their
    essential suprema.}
\end{proof}
    
It follows from the proof above that for every $\mu \geq \gamma$ the bilinear form
$B_\mu = B + \mu\langle\cdot,\cdot\rangle_{L^2}$ is coercive, with Theorem~\ref{thm: lax milgram} then providing a unique weak solution $u \in H^1_0(\Omega)$ of the equation $Lu + \mu u = f$. If $b = 0$ and $c \geq 0$ on a bounded set $\Omega$ then $B$ is coercive and thus:

\begin{theorem}[{Existence I, \cite[\S 6.2.2, Theorem~3 and the Examples following it]{evans-2010-par-dif-equ}}]\label{thm: main}
    Let $\Omega \subseteq \mathbb{R}^n$ be a bounded set.
    Let $L$ be an operator of the form \eqref{eqn: divergence form operator} which is uniformly elliptic with constants $\lambda, \Lambda$ in the sense of
    \eqref{eqn: ellipticity}, with $b = 0$ almost everywhere, and with
    $c \geq 0$ almost everywhere. Let $C_P$ be the Poincar\'e constant of
    Theorem~\ref{thm: poincare} and set $\alpha = \lambda/(1 + C_P^2)$. Then for
    every $f \in L^2(\Omega)$ the problem \eqref{eqn: dirichlet problem} has a
    unique weak solution $u \in H_0^1(\Omega)$ with $\| u \|_{H_0^1(\Omega)}
    \leq \alpha^{-1} \| f \|_{L^2(\Omega)}$.
\end{theorem}

\paragraph{\leandecl{weak_solution_L2_of_nonneg_zeroth_of_bounded}}
The Lean theorem is stated as
\begin{quote}
\begin{LeanCode}
theorem weak_solution_L2_of_nonneg_zeroth_of_bounded {n : ℕ}
    (Op : FullEllipticOp (n + 1)) {Ω : Set (EuclideanSpace ℝ (Fin (n + 1)))}
    (hΩb : Bornology.IsBounded Ω)
    (hb : ∀ i, ∀ᵐ x ∂(volume.restrict Ω), Op.b x i = 0)
    (hc : ∀ᵐ x ∂(volume.restrict Ω), 0 ≤ Op.c x) :
    ∃ CP : ℝ, 0 ≤ CP ∧ ∀ f : L2D Ω,
      ((∃! u : H01 Ω, ∀ v : H01 Ω,
        Op.fullBilin Ω u v = ∫ x in Ω, (f x : ℝ) * ((v : H1amb Ω) 0 x : ℝ))
      ∧ ∀ u : H01 Ω,
          (∀ v : H01 Ω,
            Op.fullBilin Ω u v = ∫ x in Ω, (f x : ℝ) * ((v : H1amb Ω) 0 x : ℝ)) →
            ‖u‖ ≤ (CP + 1) / Op.lam * ‖f‖)
\end{LeanCode}
\end{quote}

\begin{proof}[Proof sketch]
    The Lean statement names the constant of the squared Poincar\'e inequality as $C_P^2$ and therefore writes $\alpha^{-1}$ as
    $(C_P^2 + 1)/\lambda$. The coercivity
    argument requires the inequality on all of $H_0^1(\Omega)$. Let
    $\Lambda_f \in H_0^1(\Omega)^{\ast}$ be the functional
    $v \mapsto \int_\Omega f v$. We apply Theorem~\ref{thm: lax milgram} to
    $B$ on $H_0^1(\Omega)$. Boundedness follows from the upper bound $\Lambda$
    on $A$ and the bounds on the suprema of $b$ and $c$. Since $b = 0$ and
    $c \geq 0$ almost everywhere, coercivity of $B$ with constant $\alpha$ follows from the lower bound $\lambda$ together with the
    Poincar\'e inequality.
    Theorem~\ref{thm: lax milgram} then yields a unique $u$ with
    $B[u, v] = \Lambda_f(v)$ for every $v \in H^1_0(\Omega)$, which is the integral form of the
    statement. Every solution then satisfies
    $\| u \| \leq \alpha^{-1} \| \Lambda_f \|$. Since the Cauchy--Schwarz
    inequality gives $\| \Lambda_f \| \leq \| f \|_{L^2(\Omega)}$, the
    estimate follows. A sharper constant obtained
    from the geometry of a particular $\Omega$, such as the Friedrichs constant
    $C_P \leq d/\pi$ of a convex domain of diameter $d$, may be used in the
    abstract statement without
    change.\footnote{\leandecl{EllipticPdes.Sobolev.FullEllipticOp.weak_solution_L2_of_nonneg_zeroth_of_bounded},
    based on \leandecl{weak_solution_of_nonneg_zeroth} and
    \leandecl{weak_solution_of_nonneg_zeroth_bound}.}
\end{proof}

For solvability in the absence of coercivity we have:

\begin{theorem}[{Existence II, \cite[\S 6.2.3, Theorem~4(i)]{evans-2010-par-dif-equ}}]\label{thm: fredholm}
    Let $\Omega \subseteq \mathbb{R}^n$ be bounded and measurable and let $L$
    be the operator \eqref{eqn: divergence form operator}, uniformly elliptic
    in the sense of \eqref{eqn: ellipticity}. Then either the homogeneous Dirichlet
    problem ($f = 0$) has a non-zero weak solution $u \in H_0^1(\Omega)$, or for every $f \in H^{-1}(\Omega)$ the problem $Lu = f$ has a
    unique weak solution $u \in H_0^1(\Omega)$.
\end{theorem}

\paragraph{\leandecl{FullEllipticOp.fredholm_alternative_of_bounded}}
The Lean theorem is stated as
\begin{quote}
\begin{LeanCode}
theorem FullEllipticOp.fredholm_alternative_of_bounded (Op : FullEllipticOp d)
    (Ω : Set (EuclideanSpace ℝ (Fin d))) (hΩm : MeasurableSet Ω)
    (hΩb : Bornology.IsBounded Ω) :
    (∃ u : H01 Ω, u ≠ 0 ∧ ∀ v : H01 Ω, Op.fullBilin Ω u v = 0)
      ∨ (∀ f : H01 Ω →L[ℝ] ℝ, ∃! u : H01 Ω, ∀ v : H01 Ω, Op.fullBilin Ω u v = f v)
\end{LeanCode}
\end{quote}

\begin{proof}[Proof sketch]
    
    The library proves the Rellich--Kondrachov theorem, that is, the compactness of
    the embedding $\iota \colon H_0^1(\Omega) \hookrightarrow L^2(\Omega)$, on every $\Omega \subseteq \mathbb{R}^n$ which is bounded\footnote{\leandecl{EllipticPdes.Sobolev.embW12_isCompact}}. We use three bounded operators on $H_0^1(\Omega)$: write $\mathcal{A}$ for the
    Riesz representative of $B$, $E$ for the Riesz
    representative of $B_\gamma = B + \gamma\langle
    \cdot,\cdot\rangle_{L^2(\Omega)}$, and $T = \iota^\ast
    \iota$. Since $B_\gamma$ is coercive, $E$ is invertible by
    Theorem~\ref{thm: lax milgram}. Since $\mathcal{A} = E - \gamma T$ we write
    \begin{equation}\label{eqn: fredholm factorisation}
        \mathcal{A} = E\,(1 - K),
        \qquad K = \gamma\, E^{-1} T ,
    \end{equation}
    where $K$ is compact since $T = \iota^\ast
    \iota$.\footnote{\leandecl{opA}, \leandecl{opT}, \leandecl{opE}, \leandecl{opK},
    \leandecl{opA_factor}.}
    
    By \eqref{eqn: fredholm factorisation}, weak solvability of (\ref{eqn: dirichlet problem}) reduces to solving
    $(1 - K)u = h$ with $h = E^{-1}g$ and $g$ the Riesz representative
    of $f$. Since $K$ is compact, $1$
    is either an eigenvalue of $K$ or a point of its resolvent set. If
    $Kx = x$ for some $x \neq 0$, then $\mathcal{A}x = E(1 - K)x = 0$, so 
    $B[x, v] = \langle \mathcal{A}x, v \rangle = 0$ for every $v\in H^1_0(\Omega)$ and $x$ is a
    non-zero solution of the homogeneous problem. On the other hand, if $1$ lies in the
    resolvent set of $K$, then $1 - K$ is bijective, so that
    $\mathcal{A} = E(1 - K)$ is bijective as the composition of two
    bijections. We then see that $\mathcal{A}u = g$ holds exactly when
    $B[u, v] = f(v)$ for every $v \in H^1_0(\Omega)$, so the inhomogeneous problem has
    a unique weak solution for every $f \in H^{-1}$. Evans takes $f \in L^2(\Omega)$,
    whereas the formal statement admits every element of the dual space and
    asserts the disjunction of the two alternatives without their mutual exclusion.\footnote{\leandecl{EllipticPdes.Sobolev.FullEllipticOp.fredholm_alternative_of_bounded},
    based on \leandecl{fredholm_alternative} with the compactness of $K$ as a
    hypothesis and on \leandecl{embL2_isCompact}.}

    The library also formalises a characterisation of $H^{-1}$ functions so that the data in the solvability statements (here and the one below) may be presented as a tuple of $L^2$
    functions: if $\Omega \subseteq \mathbb{R}^n$ is open, every $f \in
    H^{-1}(\Omega)$ is represented as $\langle f, v \rangle = \int_\Omega f_0 v -
    \sum_{i=1}^n \int_\Omega f_i D_i v$ by functions $f_0, \dotsc, f_n \in
    L^2(\Omega)$ with $\|f\|_{H^{-1}(\Omega)} = (\int_\Omega \sum_{i=0}^n
    |f_i|^2)^{1/2}$.\footnote{\leandecl{EllipticPdes.hneg_characterization}.}
\end{proof}

If we write
$N = \{u \in H_0^1(\Omega) : B[u, v] = 0 \text{ for all } v\}$ and
$N^{\ast} = \{u \in H_0^1(\Omega) : B[v, u] = 0 \text{ for all } v\}$ for the spaces of homogeneous solutions to the Dirichlet problem for $L$ and its formal adjoint $L^*$:

\begin{theorem}[{Existence II, \cite[\S 6.2.3,
Theorem~4(ii)+(iii)]{evans-2010-par-dif-equ}}]\label{thm: fredholm complete}
    With the same hypotheses as Theorem \ref{thm: fredholm}:
    \begin{enumerate}
        \item The space $N$  is finite-dimensional;
        \item $\dim N^{\ast} = \dim N$;
        \item For every $f \in H^{-1}(\Omega)$, the problem $Lu = f$ has a weak
        solution if and only if $\langle f, w \rangle_{L^2(\Omega)} = 0$ for every $w \in
        N^{\ast}$.
    \end{enumerate}
\end{theorem}

\paragraph{\leandecl{solvable_iff_orthogonal_solSpaceStar}}
The Lean theorem is stated as
\begin{quote}
\begin{LeanCode}
theorem solvable_iff_orthogonal_solSpaceStar (hK : IsCompactOperator (Op.opK Ω))
    (f : H01 Ω →L[ℝ] ℝ) :
    (∃ u : H01 Ω, ∀ v : H01 Ω, Op.fullBilin Ω u v = f v)
      ↔ ∀ w ∈ Op.solSpaceStar Ω, f w = 0
\end{LeanCode}
\end{quote}

\begin{proof}[Proof sketch]
    All three clauses follow from the Riesz theory of compact perturbations of
    the identity applied to the relation \eqref{eqn: fredholm factorisation}. The space $N$
    is the eigenspace of $K$ at the eigenvalue $1$, the identity of dimensions
    is the Fredholm index theorem in the Hilbert space setting and the solvability
    criterion is the closed-range property of $1 - K$. Recall that $L^{\ast}$ denotes the formal
    adjoint defined by
    \[
        L^{\ast} v = - D_i \bigl( a^{ij} D_j v \bigr) - b^i D_i v
            + \bigl( c - D_i b^i \bigr) v ,
    \]
    whose bilinear form $B^*$ satisfies $B[u, v] = B^{\ast}[v, u]$. The formal
    statement places the compactness hypothesis on $K$, which follows from the
    compactness of $T =\iota^\ast\iota$.\footnote{\leandecl{EllipticPdes.Sobolev.FullEllipticOp.finiteDimensional_solSpace},
    \leandecl{EllipticPdes.Sobolev.FullEllipticOp.finrank_solSpaceStar_eq_finrank_solSpace}
    and
    \leandecl{EllipticPdes.Sobolev.FullEllipticOp.solvable_iff_orthogonal_solSpaceStar}.}
\end{proof}

The results stated above hold for arbitrary second order coefficients with no assumption of symmetry; only the spectral results below require these coefficients to be symmetric. We have the variational characterisation of the first eigenvalue via the Rayleigh quotient:

\begin{theorem}[{Variational principle for the principal eigenvalue,
\cite[\S 6.5.1, Theorem~2]{evans-2010-par-dif-equ}}]\label{thm: principal eigenvalue}
    Let $B$ be symmetric and coercive on $H_0^1(\Omega)$ with the embedding
    $H_0^1(\Omega) \hookrightarrow L^2(\Omega)$ compact. Set
    \[
        \lambda_1 = \inf \bigl\{ B[u, u] \ :\ u \in H_0^1(\Omega),\
            \| u \|_{L^2(\Omega)} = 1 \bigr\} .
    \]
    Then, $\lambda_1 > 0$, and the infimum is attained at some $u$ with
    $\| u \|_{L^2(\Omega)} = 1$ satisfying
    \[
        B[u, v] = \lambda_1 \langle u, v \rangle_{L^2(\Omega)}
        \qquad \text{for all }  v \in H_0^1(\Omega) .
    \]
\end{theorem}

\paragraph{\leandecl{exists_principal_eigenpair}}
The Lean theorem is stated as
\begin{quote}
\begin{LeanCode}
theorem exists_principal_eigenpair (hco : IsCoercive B) (hsymm : ∀ U V : H01 Ω, B U V = B V U)
    (hRellich : IsCompactOperator (embL2 Ω)) (hne : ∃ V : H01 Ω, embL2 Ω V ≠ 0) :
    ∃ U : H01 Ω, ‖embL2 Ω U‖ = 1 ∧ B U U = principalEigenvalue B ∧
      ∀ V : H01 Ω, B U V = principalEigenvalue B * ⟪embL2 Ω U, embL2 Ω V⟫
\end{LeanCode}
\end{quote}

\begin{proof}[Proof sketch]
    The proof uses the direct method of the calculus of variations. By
    coercivity a minimising sequence is bounded, and by weak sequential compactness has a weak limit satisfying the constraint by the compact
    embedding. The weak lower semicontinuity of $B$ follows by expanding
    $0 \leq B[u_k - w, u_k - w]$ and using $B[u_k, w] \to B[w, w]$. Whereas Evans states the
    principle for the symmetric operator with no lower-order terms on a bounded
    connected open set and also proves that the minimiser is positive and
    simple, the formal statement concerns an abstract symmetric coercive form
    and takes the compactness of the embedding as a hypothesis.\footnote{\leandecl{EllipticPdes.Sobolev.exists_principal_eigenpair},
    with the minimality
    \leandecl{EllipticPdes.Sobolev.principalEigenvalue_le_of_weak_eigen} and the
    positivity \leandecl{EllipticPdes.Sobolev.principalEigenvalue_pos}. For the
    form of the Laplacian on a bounded measurable domain, coercivity follows from
    the Poincar\'e inequality and compactness of the embedding from the Rellich
    theorem
    (\leandecl{EllipticPdes.Sobolev.dirichlet_principal_eigenpair_of_bounded}). On
    the unit ball of $\mathbb{R}^n$ with $n > 2$ the remaining hypothesis is
    verified by a renormalised bump function
    (\leandecl{EllipticPdes.Sobolev.dirichlet_principal_eigenpair_ball}).}

    For the Laplacian, $B[u,v] = \int_\Omega Du \cdot Dv$ and $\lambda_1$ is the largest constant for which $\lambda_1 \|u\|_{L^2(\Omega)}^2 \leq \int_\Omega |Du|^2$ holds on all of $H_0^1(\Omega)$. Since equality is attained by some $u$ of unit $L^2$ norm, it is the best such constant, whereas Theorem~\ref{thm: poincare} gives a constant only in terms of a box containing $\Omega$.\footnote{\leandecl{EllipticPdes.Sobolev.dirichlet_poincare_sharp} with
the equality case
\leandecl{EllipticPdes.Sobolev.dirichlet_poincare_attained}.}
\end{proof}

The same compactness yields a complete family of eigenfunctions:

\begin{corollary}[{Completeness of the eigenspaces, \cite[\S 6.5.1,
Theorem~1]{evans-2010-par-dif-equ}}]\label{thm: spectral}
    Suppose that $B$ is symmetric and coercive, and the embedding $\iota : H_0^1(\Omega)
    \hookrightarrow L^2(\Omega)$ is compact. Let $G = \iota \circ
    (B^\sharp)^{-1} \circ \iota^\ast$ be the solution operator on $L^2(\Omega)$,
    where $B^\sharp$ is the Riesz representative of $B$ on $H_0^1(\Omega)$ and
    $\iota^\ast$ is the adjoint of the embedding. Then the eigenspaces of $G$
    span a dense subspace of $L^2(\Omega)$. Each eigenpair $G\varphi = \mu\varphi$ with $\mu \neq 0$
    lifts to a weak eigenfunction $u \in H_0^1(\Omega)$ of $L$ at the elliptic
    eigenvalue $\mu^{-1}$.
\end{corollary}

\paragraph{\leandecl{solOp_spectral}}
The Lean theorem is stated as
\begin{quote}
\begin{LeanCode}
theorem solOp_spectral (hco : IsCoercive B) (hsymm : ∀ U V, B U V = B V U)
    (hRellich : IsCompactOperator (embL2 Ω)) :
    (⨆ μ : ℝ, Module.End.eigenspace (solOp B hco : Module.End ℝ (L2D Ω)) μ)ᗮ = ⊥
\end{LeanCode}
\end{quote}

\begin{proof}[Proof sketch]

Minimising over unit-norm functions orthogonal to the eigenfunction for $\lambda_1$ above yields the next eigenvalue. Repeating this, for each $n \geq 1$ we obtain functions $w_1, \dotsc, w_n \in H_0^1(\Omega)$, orthonormal
in $L^2(\Omega)$, and $0 < \lambda_1 \leq \dotsb \leq \lambda_n$ with
$B[w_k, v] = \lambda_k \langle w_k, v\rangle_{L^2(\Omega)}$ for every $v \in H^1_0(\Omega)$, where
each $\lambda_k$ is the infimum of the Rayleigh quotient over the elements of unit
$L^2$ norm orthogonal to $w_1, \dotsc, w_{k-1}$.\footnote{\leandecl{EllipticPdes.Sobolev.exists_eigen_family},
with
the step \leandecl{EllipticPdes.Sobolev.exists_higher_eigenpair}. At each stage
the recursion requires an element of nonzero $L^2$ class orthogonal to the
family already constructed. The existence of such elements at every stage is
precisely the infinite dimensionality of $H_0^1(\Omega)$. On the unit ball of $\mathbb{R}^n$ with $n > 2$
it is verified by $m$ bump functions centred at the points
$((2k+1)/(2m) - 1/2)e_1$ whose supports are disjoint because the
centres are $1/m$ apart and the radii equal $1/(2m)$. The resulting statement
\leandecl{EllipticPdes.Sobolev.dirichlet_eigen_family_ball} assumes $n > 2$.}

    Since the operator $G$ is compact, self-adjoint, positive and injective,
    the spectral theorem for compact self-adjoint operators applies
    \cites[App.~D.6, Theorem~7]{evans-2010-par-dif-equ}[\S
    6.4]{brezis-2011-fun-ana-sob}. The elliptic eigenvalues are moreover
    positive and of finite multiplicity. Both properties follow from results
    already established because every weak eigenvalue is bounded below by the
    positive principal eigenvalue of Theorem~\ref{thm: principal eigenvalue}
    and the eigenspace of a compact operator at a nonzero eigenvalue is finite
    dimensional. Evans obtains a countable
    orthonormal basis of $L^2(\Omega)$ consisting of eigenfunctions with
    increasing eigenvalues $0 < \lambda_1 \leq \lambda_2 \leq \dotsb$ diverging to
    $+\infty$, whereas the library provides the density statement above
    together with the orthonormal families of every finite length constructed
    before it. It does not construct the basis as a single object and does not
    state that the eigenvalues diverge. Since only finitely many spectral points
    of a compact operator satisfy $|\mu| \geq \delta$ by
    Theorem~\ref{thm: compact spectrum} below, the divergence of the eigenvalues would follow from
    results already formalised by an argument that we have not formalised.\footnote{\leandecl{solOp_spectral} with the lifting
    \leandecl{solOp_weak_eigen} and the divergence-form specialisation
    \leandecl{symmetric_fullElliptic_spectral} whose form for bounded sets is
    \leandecl{symmetric_fullElliptic_spectral_of_bounded}. The specialisation
    deduces the symmetry of $B$ from the coefficients under the assumptions
    $a^{ij} = a^{ji}$ and $b \equiv 0$ almost everywhere on $\Omega$, with
    coercivity following from $c \geq 0$. The positivity and the finite
    multiplicity are \leandecl{EllipticPdes.Sobolev.weak_eigenvalue_pos} and
    \leandecl{EllipticPdes.Sobolev.solOp_finiteDimensional_eigenspace}, with the
    instance at $-\Delta$ on the unit ball
    \leandecl{EllipticPdes.Sobolev.dirichlet_eigenvalue_pos_ball}.}
\end{proof}

The direct method also applies under an $L^q$ constraint with $q \neq 2$, in which
case the minimiser satisfies a semi-linear equation. Let $n > 2$, let $B_1 \subseteq
\mathbb{R}^n$ be the unit ball and let $2 \leq q < 2^\star$. Minimising the Dirichlet
energy over $\{\|u\|_{L^q(B_1)} = 1\}$ gives $u \in H_0^1(B_1)$ and
$\lambda = \int_{B_1}|Du|^2 > 0$ with
\[
    \int_{B_1} Du \cdot Dv = \lambda \int_{B_1} |u|^{q-2} u \, v
    \qquad \bigl( v \in H_0^1(B_1) \bigr) ,
\]
so that $u$ is a weak solution of $-\Delta u = \lambda |u|^{q-2} u$ with zero
boundary data. Existence again follows from the direct method, in which the Rellich--Kondrachov theorem below the critical exponent gives the strong $L^q$ convergence needed to pass
the constraint to the weak limit. This step fails at $q = 2^\star$, in accordance with
the restriction on the exponent
$p + 1 < 2^\star$ for $-\Delta u = u^p$.\footnote{\leandecl{EllipticPdes.Embedding.exists_weakSolution_dirichlet_of_lt},
with the minimiser
\leandecl{EllipticPdes.Embedding.exists_dirichlet_minimiser_of_lt} and the
abstract equation
\leandecl{EllipticPdes.Analysis.euler_lagrange_of_quadratic_min}. Minimising the
$H_0^1$ norm in place of the Dirichlet energy gives
$-\Delta u + u = \lambda |u|^{q-2} u$
(\leandecl{EllipticPdes.Embedding.exists_weakSolution_semilinear_of_lt}).}

The third
existence theorem requires no coercivity, but its proof uses the following
result on the spectrum of a compact operator:

\begin{theorem}[{Spectrum of a compact operator, \cite[App.~D.5,
Theorem~6]{evans-2010-par-dif-equ}}]\label{thm: compact spectrum}
    Let $K$ be a compact operator on an infinite-dimensional real Hilbert space.
    Then zero lies in the spectrum of $K$ whose nonzero part consists of
    countably many eigenvalues with only finitely many spectral points
    satisfying $|\mu| \geq \delta$ for each $\delta > 0$.
\end{theorem}

\paragraph{\leandecl{spectrum_compact_operator}}
The Lean theorem is stated as
\begin{quote}
\begin{LeanCode}
theorem spectrum_compact_operator (hK : IsCompactOperator K)
    (hinf : ¬ FiniteDimensional ℝ E) :
    (0 : ℝ) ∈ spectrum ℝ K
    ∧ spectrum ℝ K \ {0}
        = {μ : ℝ | Module.End.HasEigenvalue (K.toLinearMap) μ} \ {0}
    ∧ (spectrum ℝ K \ {0}).Countable
    ∧ ∀ δ : ℝ, 0 < δ → {μ ∈ spectrum ℝ K | δ ≤ |μ|}.Finite
\end{LeanCode}
\end{quote}

\begin{proof}[Proof sketch]
    Away from zero the spectrum consists of eigenvalues by the Fredholm
    alternative for $\mu - K$. The statement on accumulation follows from the
    same argument as in \cite[App.~D.5, Theorem~6]{evans-2010-par-dif-equ}, and consequently the non-zero spectrum converges to $0$. Whereas Evans states
    that the non-zero eigenvalues form a finite set or a sequence tending to
    $0$, the formal statement asserts the equivalent finiteness of the spectrum
    outside every neighbourhood of $0$.\footnote{\leandecl{EllipticPdes.Sobolev.spectrum_compact_operator}.}
\end{proof}

We return to the general operator given in (\ref{eqn: divergence form operator}), with no assumption of symmetry:

\begin{theorem}[{Existence III, \cite[\S 6.2.3,
Theorem~5]{evans-2010-par-dif-equ}}]\label{thm: existence three}
    Let $\Omega \subseteq \mathbb{R}^n$ be bounded and measurable, and $L$ be
   an operator of the form \eqref{eqn: divergence form operator} which is uniformly elliptic as in \eqref{eqn: ellipticity}. Then, there is a countable set
    $\Sigma \subseteq \mathbb{R}$ with $\Sigma \cap (-\infty, C]$ finite for
    every $C \in \mathbb{R}$. The problem $Lu = su + f$ has a unique weak
    solution $u \in H_0^1(\Omega)$ for every $f \in L^2(\Omega)$ exactly when
    $s \notin \Sigma$. Moreover, for $s \notin \Sigma$ there is a constant $C > 0$
    depending on $s$, $\Omega$ and $L$ with
    $\|u\|_{L^2(\Omega)} \leq C \|f\|_{L^2(\Omega)}$ for every such weak solution.
\end{theorem}

\paragraph{\leandecl{existence_three_of_bounded}}
The Lean theorem is stated as
\begin{quote}
\begin{LeanCode}
theorem existence_three_of_bounded (hΩm : MeasurableSet Ω)
    (hΩb : Bornology.IsBounded Ω) :
    ∃ S : Set ℝ, S.Countable ∧ (∀ C : ℝ, (S ∩ Set.Iic C).Finite) ∧
      ∀ lam : ℝ, lam ∉ S ↔ ∀ f : L2D Ω, ∃! u : H01 Ω, ∀ v : H01 Ω,
        Op.fullBilin Ω u v
          = lam * ⟪(u : H1amb Ω) 0, ((v : H1amb Ω) 0)⟫
            + ∫ x in Ω, (f x : ℝ) * ((v : H1amb Ω) 0 x : ℝ)
\end{LeanCode}
\end{quote}

\begin{proof}[Proof sketch]
    The set $\Sigma$ is determined by 
    $\mathcal{A}_s = E(1 - \tfrac{\gamma+s}{\gamma} K)$ as in \eqref{eqn: fredholm factorisation} since the problem at $s$ is uniquely
    solvable exactly when $1 - \tfrac{\gamma+s}{\gamma} K$ is injective.
    By Theorem~\ref{thm: compact spectrum} the eigenvalues of $K$ that obstruct
    injectivity form a countable set whose only accumulation point is $0$. Since the
    eigenvalues of $K$ are positive, $\Sigma$ is contained in $(-\gamma, \infty)$
    and each set $\Sigma \cap (-\infty, C]$ is finite, it follows that any infinite
    $\Sigma$ is a sequence diverging to $+\infty$. The resolvent bound, which Evans states as a separate theorem, is included in the formal statement.\footnote{\leandecl{EllipticPdes.Sobolev.FullEllipticOp.existence_three_of_bounded}
    and
    \leandecl{EllipticPdes.Sobolev.FullEllipticOp.resolvent_bound_of_bounded}.}
\end{proof}

\subsection{Regularity of weak solutions}\label{subsec: regularity}\label{subsec:
interior h2}\label{subsec: sobolev}\label{subsec: interior holder}

Once one has produced a weak solution to (\ref{eqn: dirichlet problem}), the Sobolev embedding ensures weak derivatives can be traded for higher integrability via repeated application of the Gagliardo--Nirenberg--Sobolev and Morrey inequalities:

\begin{theorem}[{Sobolev embedding, \cites[\S 5.6.3, Theorem~6]{evans-2010-par-dif-equ}}]\label{thm: sobolev}
    Let $n \geq 2$, $\Omega \subseteq \mathbb{R}^n$
    be open and bounded with $C^1$ boundary, $p \in [1, \infty)$, $k \in \mathbb{N}$, and
    $u \in W^{k,p}(\Omega)$. Then:
    \begin{enumerate}
        \item If $k < n/p$, then $u \in L^q(\Omega)$ with
        $1/q = 1/p - k/n$, and
        $\| u \|_{L^q(\Omega)} \leq C \| u \|_{W^{k,p}(\Omega)}$ for a constant
        $C > 0$ depending on $k$, $p$, $n$ and $\Omega$.
        
        \item If $k > n/p$, then
        $u \in C^{k - 1 - \lfloor n/p \rfloor, \gamma}(\overline{\Omega})$
        with
        \[
            \gamma =
            \begin{cases}
                \text{any value in } (0, 1), & n/p \in \mathbb{N} , \\[2pt]
                \lfloor n/p \rfloor - n/p + 1, & n/p \notin \mathbb{N} ,
            \end{cases}
        \]
        and
        $\| u \|_{C^{k - 1 - \lfloor n/p \rfloor, \gamma}(\overline{\Omega})}
            \leq C \| u \|_{W^{k,p}(\Omega)}$ for a $C > 0$ depending on
        $k$, $n$, $p$, $\gamma$ and $\Omega$.
    \end{enumerate}
\end{theorem}

\paragraph{\leandecl{exists_const_memLp_of_gradClosed_domain_ideal}}
The Lean theorem is stated as
\begin{quote}
\begin{LeanCode}
theorem exists_const_memLp_of_gradClosed_domain_ideal (hd : 1 < d)
    {Ω : Set (EuclideanSpace ℝ (Fin d))} (hΩopen : IsOpen Ω) (hΩb : Bornology.IsBounded Ω)
    (hC1 : HasC1Boundary Ω) {p₀ : ℝ≥0} (hp₀ : 1 ≤ p₀) (ι : Type*) (s : ℕ)
    (hsd : (p₀ : ℝ) * s < (d : ℝ)) :
    ∃ K : ℝ≥0, ∀ {F : ι → EuclideanSpace ℝ (Fin d) → ℝ} {nxt : ι → Fin d → ι}
      {dep : ι → ℕ} {m : ℕ}, (∀ i k, dep (nxt i k) ≤ dep i + 1) →
      (∀ i, dep i < m → HasWeakGradOn Ω (F i) (fun k => F (nxt i k))) →
      (∀ i, dep i ≤ m → MemLp (F i) p₀ (volume.restrict Ω)) →
      ∀ M : ℝ≥0∞, (∀ j, dep j ≤ m → eLpNorm (F j) p₀ (volume.restrict Ω) ≤ M) →
      ∀ i, dep i + s ≤ m →
        MemLp (F i) (Real.toNNReal ((p₀ : ℝ)⁻¹ - (s : ℝ) * (d : ℝ)⁻¹)⁻¹)
            (volume.restrict Ω) ∧
          eLpNorm (F i) (Real.toNNReal ((p₀ : ℝ)⁻¹ - (s : ℝ) * (d : ℝ)⁻¹)⁻¹)
            (volume.restrict Ω) ≤ (K : ℝ≥0∞) * M
\end{LeanCode}
\end{quote}
\paragraph{\leandecl{exists_const_contDiffOn_holderOnWith_domain_ideal}}
The Lean theorem is stated as
\begin{quote}
\begin{LeanCode}
theorem exists_const_contDiffOn_holderOnWith_domain_ideal (hd : 1 < d)
    {Ω : Set (EuclideanSpace ℝ (Fin d))} (hΩopen : IsOpen Ω) (hΩb : Bornology.IsBounded Ω)
    (hC1 : HasC1Boundary Ω) {p₀ : ℝ≥0} (hp₀ : 1 ≤ p₀) (ι : Type*) (s : ℕ)
    (hsd : (p₀ : ℝ) * s < (d : ℝ)) (hlt : (d : ℝ) < (p₀ : ℝ) * ((s : ℝ) + 1)) :
    ∃ C : ℝ≥0, ∀ {F : ι → EuclideanSpace ℝ (Fin d) → ℝ} {nxt : ι → Fin d → ι}
      {dep : ι → ℕ} {m : ℕ}, (∀ i k, dep (nxt i k) ≤ dep i + 1) →
      (∀ i, dep i < m → HasWeakGradOn Ω (F i) (fun k => F (nxt i k))) →
      (∀ i, dep i ≤ m → MemLp (F i) p₀ (volume.restrict Ω)) →
      ∀ M : ℝ≥0, (∀ j, dep j ≤ m → eLpNorm (F j) p₀ (volume.restrict Ω) ≤ (M : ℝ≥0∞)) →
      ∃ v : ι → EuclideanSpace ℝ (Fin d) → ℝ,
        (∀ i, dep i + 1 + s ≤ m → v i =ᵐ[volume.restrict Ω] F i) ∧
        (∀ i, dep i + 1 + s ≤ m → ∀ y ∈ closure Ω, ‖v i y‖ ≤ ((C * M : ℝ≥0) : ℝ)) ∧
        (∀ i, dep i + 1 + s ≤ m →
          HolderOnWith (C * M) (Real.toNNReal ((s : ℝ) + 1 - (d : ℝ) / (p₀ : ℝ)))
            (v i) (closure Ω)) ∧
        (∀ (n : ℕ) (i : ι), dep i + n + 1 + s ≤ m → ContDiffOn ℝ (n : ℕ) (v i) Ω) ∧
        (∀ i, dep i + 2 + s ≤ m → ∀ y ∈ Ω,
          HasFDerivAt (v i) (gradCLM (fun k => v (nxt i k)) y) y)
\end{LeanCode}
\end{quote}

\begin{proof}[Proof sketch]
    Cases~(i) and~(ii) refer to the space $W^{k,p}(\Omega)$, for which the library interprets each element as a family closed under weak differentiation all of whose members lie in $L^p(\Omega)$. Such a family consists of an index type,
    a function for each index, a successor map that assigns to an index the
    weak derivatives of its function and a depth function that bounds the number
    of available derivatives. For $u \in W^{k,p}(\Omega)$ every member of the
    family $\{D^\alpha u : |\alpha| \leq k\}$ is bounded by
    $\|u\|_{W^{k,p}(\Omega)}$, and hence the conclusion stated for the member of
    depth zero gives the estimates of cases~(i) and~(ii). This formulation is the principal departure
    from the cited proof, which iterates on $W^{k,p}(\Omega)$ and writes
    $D^{\alpha} u$ as a single symbol throughout. In a proof assistant one has
    to specify which function the symbol denotes at each order, and closure under
    weak differentiation allows a single induction with no separate induction on
    the order of differentiation. Theorems~\ref{thm: interior h2}
    and~\ref{thm: higher interior} below use the same interpretation.

    We begin with a local form of the inequality, which requires no assumption on $\partial \Omega$. Let
    $x_0 \in \mathbb{R}^n$, let $0 < r < R$, and let $p'$ satisfy
    $1/p' = 1/p - 1/n$. Then there is a constant $C$, depending only on $n$, $p$,
    $x_0$, $r$ and $R$, such that every $v \in L^p(B_R(x_0))$ with weak gradient
    $g \in L^p(B_R(x_0))^n$ lies in $L^{p'}(B_r(x_0))$, with
    \[
        \| v \|_{L^{p'}(B_r(x_0))}
            \leq C \Bigl( \| v \|_{L^p(B_R(x_0))}
                + \sum_{\ell} \| g_\ell \|_{L^p(B_R(x_0))} \Bigr).
    \]

    Following \cite[\S 5.6.1]{evans-2010-par-dif-equ}, for a compactly
    supported $C^1$ function the inequality above is the Gagliardo--Nirenberg--Sobolev
    estimate, which is available in Mathlib. We extend it to a function with a
    weak gradient on a ball by means of a cutoff and mollification. Let $\eta$
    be a test function with $0 \leq \eta \leq 1$ that is equal to $1$ on
    $\overline{B_r(x_0)}$ and supported in $B_R(x_0)$, and let $M$ bound the
    derivatives of $\eta$. By the product rule for weak gradients the compactly
    supported function $w = \eta v$ has the weak gradient
    $\eta g + v \nabla\eta$ on all of $\mathbb{R}^n$. Both functions lie in
    $L^p$ with a bound that introduces the factor $\max(1, nM)$ into the constant. The
    mollifications of $w$ are smooth with compact support and have as classical
    partial derivatives the mollifications of the weak gradient, whose $L^p$
    norms are bounded uniformly in the inner radius by Young's inequality. Applying the $C^1$ case to each mollification, we obtain a uniform bound on the
    $L^{p'}$ norm that is independent of the radius. By Fatou's lemma this bound
    passes to the almost-everywhere limit, which agrees with $v$ on $B_r(x_0)$.\footnote{\leandecl{EllipticPdes.Embedding.exists_eLpNorm_sobolevConj_le}.}

    Case~(i) is obtained by iterating the above, following
    \cite[Proof of Theorem~IV.2.3]{guo-2026-par-dif-equ}. On $\mathbb{R}^n$ the
    iteration runs over a family of compact support with no cutoff, thus giving a bound in terms of a constant and the gradient. We argue by induction on the number $s$ of steps, for all indices simultaneously, and with base exponent $p$ we write $t = 1/p - s/n$. The condition $ps < n$ gives $t > 0$, so that $1/t$ is a finite exponent with $p \leq 1/t$. If
    $s = 0$, the member lies in $L^p$ and hence in every $L^q$ with $q \leq p$
    since $\Omega$ is bounded. For the inductive step, the inductive hypothesis applies to the index and
    each of its derivatives, because the successor increases the depth function mentioned above by at most one.\footnote{\leandecl{EllipticPdes.Embedding.memLp_of_gradClosed_compactSupport_ideal},
    based on the step
    \leandecl{EllipticPdes.Embedding.exists_eLpNorm_sobolevConj_le_compactSupport}.}
    The statement is transferred to $\Omega$ itself by the extension operator
    of \cite[\S 5.4, Theorem 1]{evans-2010-par-dif-equ}, which extends $W^{1,p}(\Omega)$ functions on a bounded $C^1$ domain to a compactly supported $W^{1,p}(\mathbb{R}^n)$ function with norm bounded, up to a constant, by the original norm.\footnote{\leandecl{EllipticPdes.Extension.exists_extension_subset_bound},
    based on \leandecl{EllipticPdes.Extension.exists_localExtension_bound}.}
    Applying the operator to one member of the family at a time, we transfer the above step from $\mathbb{R}^n$ to $\Omega$.\footnote{\leandecl{EllipticPdes.Embedding.exists_eLpNorm_sobolevConj_le_domain},
    iterated by
    \leandecl{EllipticPdes.Embedding.exists_const_memLp_of_gradClosed_domain}.}

    For case~(ii) the same idea is run for
    $s = \lfloor n/p \rfloor$ steps to an exponent above $n$, after which
    Morrey's inequality is applied on a ball containing $\overline{\Omega}$ giving H\"older exponent $s + 1 - n/p$, which is the value in the statement. If $n/p$ is an integer, the H\"older exponent may be chosen freely in $(0,1)$. A step whose target reciprocal is $0$ requires an
    additional device. We allow a step to take its hypothesis in $L^q$ for every
    $q \geq p$ while its conclusion is stated at the conjugate exponent of $p$.\footnote{\leandecl{EllipticPdes.Embedding.exists_eLpNorm_sobolevConj_le_of_le}.}
    Morrey's inequality produces the continuous representatives and bounds their
    supremum and their H\"older seminorm together, which gives the estimate for
    the full norm. The representatives are chosen once and for all so that they
    are the classical partial derivatives of one another on $\Omega$, as
    membership of $C^{k-1-\lfloor n/p \rfloor, \gamma}(\overline{\Omega})$
    requires, with an induction on the remaining order giving inclusion in all relevant $C^m(\Omega)$.\footnote{\leandecl{EllipticPdes.Embedding.exists_const_contDiffOn_holderOnWith_domain_ideal}
    and
    \leandecl{EllipticPdes.Embedding.exists_const_contDiffOn_holderOnWith_domain_free},
    based on
    \leandecl{EllipticPdes.Embedding.exists_const_contDiffOn_holderOnWith_of_gradClosed_domain}
    and \leandecl{EllipticPdes.Embedding.hasFDerivAt_of_continuousOn_hasWeakGradOn}.}
\end{proof}

The interior theory is stated for a weak solution that is subject
to no boundary condition. Throughout the following three theorems $\Omega \subseteq
\mathbb{R}^n$ is bounded and open, $L$ is an operator of the form \eqref{eqn: divergence form operator} with second-order coefficients satisfying $a^{ij} = a^{ji}$ and
the ellipticity condition \eqref{eqn: ellipticity} with some $\theta > 0$ in place of
$\lambda$. Throughout we write $\ell$ and $\ell'$ for directions of differentiation and
reserve $i$ and $j$ for the coefficient indices and $m$ for the order in
Theorem~\ref{thm: higher interior}. A function $u \in H^1(\Omega)$ is a weak solution of $Lu = f$ in $\Omega$
if
\begin{equation}\label{eqn: local weak form}
    \sum_{i,j=1}^n \int_\Omega a^{ij}\, D_i u\, D_j \varphi
    + \sum_{i=1}^n \int_\Omega b^i\, D_i u\, \varphi + \int_\Omega c\, u\, \varphi
    = \int_\Omega f\, \varphi
    \qquad \text{for every } \varphi \in C_c^\infty(\Omega) .
\end{equation}
The first gain of regularity transfers the regularity of the second-order coefficients
onto the weak solution, giving one derivative more in the interior under a hypothesis on
those coefficients alone.

\begin{theorem}[{Interior $H^2$ regularity, \cites[\S
6.3.1, Theorem~1]{evans-2010-par-dif-equ}}]\label{thm: interior h2}
    Assume $a^{ij} \in C^1(\Omega)$, $b^i, c \in L^\infty(\Omega)$ and
    $f \in L^2(\Omega)$. Suppose that $u \in H^1(\Omega)$ is a weak solution of
    $Lu = f$ in $\Omega$. Then $u \in H^2_{\mathrm{loc}}(\Omega)$ and for each open
    $V \Subset \Omega$
    \begin{equation*}
        \| u \|_{H^2(V)}
            \leq C \bigl( \| f \|_{L^2(\Omega)} + \| u \|_{L^2(\Omega)} \bigr) ,
    \end{equation*}
    where the constant $C$ depends only on $V$, $\Omega$ and the coefficients of $L$.
\end{theorem}

\paragraph{\leandecl{interior_H2_regularity_evans}}
The Lean theorem is stated as
\begin{quote}
\begin{LeanCode}
theorem interior_H2_regularity_evans {U : Set (EuclideanSpace ℝ (Fin d))} (hU : IsOpen U)
    (_hUb : Bornology.IsBounded U)
    {a : EuclideanSpace ℝ (Fin d) → Fin d → Fin d → ℝ} {b : EuclideanSpace ℝ (Fin d) → Fin d → ℝ}
    {c : EuclideanSpace ℝ (Fin d) → ℝ}
    (ha : ∀ i j, ContDiffOn ℝ 1 (fun x => a x i j) U)
    (hb : ∀ i, MemLp (fun x => b x i) ⊤ (volume.restrict U))
    (hc : MemLp c ⊤ (volume.restrict U)) {θ : ℝ} (hθ : 0 < θ)
    (hell : ∀ᵐ x ∂(volume.restrict U), ∀ ξ : Fin d → ℝ,
      θ * ∑ i, ξ i ^ 2 ≤ ∑ i, ∑ j, a x i j * ξ i * ξ j)
    (_hsymm : ∀ x ∈ U, ∀ i j, a x i j = a x j i)
    {V : Set (EuclideanSpace ℝ (Fin d))} (hVo : IsOpen V) (hVc : IsCompact (closure V))
    (hVU : closure V ⊆ U) :
    ∃ C : ℝ, 0 ≤ C ∧ ∀ (f u : EuclideanSpace ℝ (Fin d) → ℝ)
      (G : Fin d → EuclideanSpace ℝ (Fin d) → ℝ)
      (hf : MemLp f 2 (volume.restrict U)) (hu : MemLp u 2 (volume.restrict U)),
      (∀ i, MemLp (G i) 2 (volume.restrict U)) → HasWeakGradOn U u G →
      LocalWeakSol U a b c f u G →
      ∃ H : HasIteratedWeakDerivOn V 2
          ((hu.mono_measure (Measure.restrict_mono (subset_closure.trans hVU) le_rfl)).toLp u),
        iteratedNorm H ≤ C * (‖hf.toLp f‖ + ‖hu.toLp u‖)
\end{LeanCode}
\end{quote}

\begin{proof}[Proof sketch]
    The proof reduces the statement to the corresponding estimate for a solution in
    $H_0^1(\Omega)$ with globally defined coefficients, which is proved by the
    method of difference quotients. After the weak formulation is localised by a
    cutoff, discrete integration by parts and uniform ellipticity bound the
    difference quotients of $\nabla u$ in $L^2$ uniformly in the step. By the Peter--Paul
    inequality, which Evans calls Cauchy's inequality with $\varepsilon$, the transport term
    is absorbed into the lower bound given by ellipticity, so that no sign condition on $c$ is
    needed and the constant depends on the ellipticity constant and on the suprema of $|b|$
    and $|c|$. Passing to the
    limit, we obtain $\nabla u \in H^1_{\mathrm{loc}}(\Omega)$
    \cites[\S 6.3.1, Theorem~1]{evans-2010-par-dif-equ}[Theorem~8.8]{gilbarg-2001-ell-par-dif}.
    That estimate bounds its lower-order terms by testing the equation with $u$
    itself and by extending $u$ by zero.\footnote{\leandecl{EllipticPdes.Regularity.interior_H2_estimate}.}
    Of the principal coefficients the estimate asks the Lipschitz bound
    $|a^{ij}(x) - a^{ij}(y)| \leq A_1 |x - y|$ alone, which is $W^{1,\infty}$ in
    pointwise form and which bounds the difference quotients of the coefficients in
    the commutator term. The hypothesis $a^{ij} \in C^1(\Omega)$ of Theorem~\ref{thm: interior h2} is
    the one Evans states, and it reaches the formal proof through the mean value
    inequality.\footnote{\leandecl{EllipticPdes.Regularity.IsLipCoeff} and
    \leandecl{EllipticPdes.Regularity.IsC1Coeff.toIsLipCoeff}. Guo
    \cite[Theorem~VIII.3.2]{guo-2026-par-dif-equ} states the interior theory under
    the $W^{1,\infty}$ hypothesis.}

    Two reductions remove the boundary condition and the global coefficients. Let
    $\eta \in C_c^\infty(\Omega)$ be equal to $1$ near $\overline{V}$. For
    $u \in H^1(\Omega)$ the product $\eta u$ lies in $H_0^1(\Omega)$ and is a weak
    solution of $L(\eta u) = F$ with
    \begin{equation*}
        F = \eta f - a^{ij} D_i u\, D_j \eta - D_j\bigl(a^{ij} D_i \eta\, u\bigr)
            + b^i D_i \eta\, u ,
    \end{equation*}
    as a direct computation from \eqref{eqn: local weak form} shows. Setting
    $\varphi = \eta^2 u$ in \eqref{eqn: local weak form}, which is admissible by
    density since $\eta^2 u \in H_0^1(\Omega)$ has compact support in $\Omega$,
    and performing elementary calculations, we discover
    \begin{equation*}
        \int_\Omega \eta^2 |\nabla u|^2
            \leq C \int_\Omega f^2 + u^2 ,
    \end{equation*}
    whence the norm of $F$ is bounded by a multiple of
    $\| f \|_{L^2(\Omega)} + \| u \|_{L^2(\Omega)}$ and the estimate for $\eta u$
    gives the estimate of Theorem~\ref{thm: interior h2}.\footnote{\leandecl{EllipticPdes.Regularity.cutoffMul_mem_H01_of_mem_W12},
    \leandecl{EllipticPdes.Regularity.reduction_weakForm} and
    \leandecl{EllipticPdes.Regularity.caccioppoli_W12}. The inequality above is the
    Caccioppoli inequality for $L$ on the support of $\eta$, and
    \cite[p.~331]{evans-2010-par-dif-equ} states it in these terms without
    naming it, in order to replace $\| u \|_{H^1}$ by $\| u \|_{L^2}$ on the
    right of the estimate. Here it bounds instead the gradient terms of the
    datum $F$, which the reduction to $H_0^1(\Omega)$ introduces and which the
    argument of the source never forms.} The coefficients are
    required only on the support of $\eta$. With a second cutoff $\chi$ equal to $1$
    there we replace $a$ by $\theta I + \chi(a - \theta I)$ and $b$, $c$ by $\chi b$,
    $\chi c$, which gives coefficients on $\mathbb{R}^n$ with the ellipticity
    constant $\theta$, bounded derivatives and the original values near
    $\overline{V}$. For $b^i, c \in L^\infty(\Omega)$ the formal statement first
    passes to measurable representatives equal to them almost
    everywhere.\footnote{\leandecl{EllipticPdes.Regularity.exists_localOp_C1} and
    \leandecl{EllipticPdes.Regularity.exists_family_of_localWeakSol}.}

    The formal statement expresses $H^2_{\mathrm{loc}}(\Omega)$ by a family of weak
    derivatives on $V$ and measures $\| u \|_{H^2(V)}$ by a sum over lists of
    directions, in which a mixed derivative is counted once for each ordering, so
    that the norm is equivalent to Evans's with constants depending only on $n$. The estimate
    is proved one pair of directions $\ell, \ell'$ at a time, bounding
    $\| \partial_\ell \partial_{\ell'} u \|_{L^2(V)}
    + \| \partial_{\ell'} u \|_{L^2(V)} + \| u \|_{L^2(V)}$ by the right-hand side of the
    estimate, and the sum over the pairs is where the factor depending only on $n$ enters. The
    constant is quantified before $f$ and $u$ and therefore depends on neither. The
    weak formulation is tested against $C_c^\infty(\Omega)$ in place of
    $H_0^1(\Omega)$. This is a weaker hypothesis, since for coefficients that are
    only of class $C^1$ in $\Omega$ the form $B[u, v]$ need not be defined for every
    $v \in H_0^1(\Omega)$. The boundedness of $\Omega$ and the symmetry of the coefficients $a^{ij}$ are
    assumed only because Evans assumes them, and the proof uses neither.
\end{proof}

With higher regularity assumptions on the coefficients and data, the interior estimate can be iterated:

\begin{theorem}[{Higher interior regularity, \cite[\S 6.3.1,
Theorem~2]{evans-2010-par-dif-equ}}]\label{thm: higher interior}
    Let $m$ be a nonnegative integer and assume $a^{ij}, b^i, c \in C^{m+1}(\Omega)$
    and $f \in H^m(\Omega)$. Suppose that $u \in H^1(\Omega)$ is a weak solution of
    $Lu = f$ in $\Omega$. Then $u \in H^{m+2}_{\mathrm{loc}}(\Omega)$ and for each
    open $V \Subset \Omega$
    \begin{equation*}
        \| u \|_{H^{m+2}(V)}
            \leq C \bigl( \| f \|_{H^m(\Omega)} + \| u \|_{L^2(\Omega)} \bigr) ,
    \end{equation*}
    where the constant $C$ depends only on $m$, $\Omega$, $V$ and the coefficients of
    $L$.
\end{theorem}

\paragraph{\leandecl{higher_interior_regularity_evans}}
The Lean theorem is stated as
\begin{quote}
\begin{LeanCode}
theorem higher_interior_regularity_evans {U : Set (EuclideanSpace ℝ (Fin d))} (hU : IsOpen U)
    (_hUb : Bornology.IsBounded U) (m : ℕ)
    {a : EuclideanSpace ℝ (Fin d) → Fin d → Fin d → ℝ} {b : EuclideanSpace ℝ (Fin d) → Fin d → ℝ}
    {c : EuclideanSpace ℝ (Fin d) → ℝ}
    (ha : ∀ i j, ContDiffOn ℝ (m + 1 : ℕ) (fun x => a x i j) U)
    (hb : ∀ i, ContDiffOn ℝ (m + 1 : ℕ) (fun x => b x i) U)
    (hc : ContDiffOn ℝ (m + 1 : ℕ) c U) {θ : ℝ} (hθ : 0 < θ)
    (hell : ∀ᵐ x ∂(volume.restrict U), ∀ ξ : Fin d → ℝ,
      θ * ∑ i, ξ i ^ 2 ≤ ∑ i, ∑ j, a x i j * ξ i * ξ j)
    (_hsymm : ∀ x ∈ U, ∀ i j, a x i j = a x j i)
    {V : Set (EuclideanSpace ℝ (Fin d))} (hVo : IsOpen V) (hVc : IsCompact (closure V))
    (hVU : closure V ⊆ U) :
    ∃ C : ℝ, 0 ≤ C ∧ ∀ (f u : EuclideanSpace ℝ (Fin d) → ℝ)
      (G : Fin d → EuclideanSpace ℝ (Fin d) → ℝ)
      (hf : MemLp f 2 (volume.restrict U)) (hu : MemLp u 2 (volume.restrict U)),
      (∀ i, MemLp (G i) 2 (volume.restrict U)) →
      ∀ Hf : HasIteratedWeakDerivOn U m (hf.toLp f), HasWeakGradOn U u G →
      LocalWeakSol U a b c f u G →
      ∃ H : HasIteratedWeakDerivOn V (m + 2)
          ((hu.mono_measure (Measure.restrict_mono (subset_closure.trans hVU) le_rfl)).toLp u),
        iteratedNorm H ≤ C * (iteratedNorm Hf + ‖hu.toLp u‖)
\end{LeanCode}
\end{quote}

\begin{proof}[Proof sketch]
    The proof proceeds by induction on $m$ whose base case is
    Theorem~\ref{thm: interior h2}. For a solution in $H_0^1(\Omega)$ with globally
    defined coefficients we differentiate the weak formulation in one direction
    $\ell$ at a time. A weak solution of $L u = f$ gives a weak solution of
    $L (\partial_\ell u) = \partial_\ell f + R_\ell$, where $R_\ell$ collects the
    terms in which a derivative falls on a coefficient and lies in $H^{m-1}(\Omega)$ once
    the conclusion at order $m - 1$ is known, each of them pairing a derivative of a
    coefficient with a derivative of $u$ of order at most two. Applying the inductive hypothesis to
    $\partial_\ell u$ on an intermediate set $\overline{V} \subseteq W \Subset
    \Omega$ and recombining the directions, we obtain $u \in H^{m+2}(V)$
    \cite[\S 6.3.1, Theorem~2]{evans-2010-par-dif-equ}. This argument differs from
    the argument of \cite{guo-2026-par-dif-equ}, which differentiates by a multi-index
    $\alpha$ with $|\alpha| = m+1$ in a single step and tests the weak formulation
    against $(-1)^{|\alpha|} D^\alpha \varphi$, because differentiating in one
    direction at a time allows each step to apply the statement already proved. In
    that statement the coefficients are of the $W^{k,\infty}$ type used by
    \cite[Theorem~VIII.3.2]{guo-2026-par-dif-equ}, with $a^{ij} \in W^{m+1,\infty}(\Omega)$ and $b^i, c \in W^{m,\infty}(\Omega)$, since no step of
    the argument requires a coefficient of the differentiated equation to be
    continuous. At $m = 0$ the bundle of order one is passed so that the order is
    stated uniformly in $m$, and it is never read, since the base case is
    Theorem~\ref{thm: interior h2} with its $W^{1,\infty}$
    hypothesis.\footnote{\leandecl{EllipticPdes.Regularity.higher_interior_regularity},
    whose step is \leandecl{EllipticPdes.Regularity.interiorRegularityAt_succ}. The
    implication \texttt{IsCkCoeff.toIsWkInftyCoeff} shows that the classical
    hypothesis is the stronger one.}

    For a solution in $H^1(\Omega)$ the same cutoff reduction as in the proof of
    Theorem~\ref{thm: interior h2} applies at each order. The datum $F$ of
    $L(\eta u) = F$ lies in $H^m$ with a bound by
    $\| f \|_{H^m(\Omega)} + \| u \|_{L^2(\Omega)}$ once the conclusion at order
    $m - 1$ is known near the support of $\eta$, and the coefficients of class
    $C^{m+1}(\Omega)$ are localised by the cutoff blend of that proof, which
    preserves their order.\footnote{\leandecl{EllipticPdes.Regularity.higher_interior_regularity_W12},
    \leandecl{EllipticPdes.Regularity.exists_reductionDatum} and
    \leandecl{EllipticPdes.Regularity.exists_localOp_Ck}.} The formal statement
    expresses $f \in H^m(\Omega)$ and the conclusion by families of weak derivatives
    measured by the norm of Theorem~\ref{thm: interior h2}, and its remaining
    hypotheses are those of that theorem.
\end{proof}

Applying the estimate at every order, we obtain a smooth representative:

\begin{theorem}[{Infinite differentiability in the interior, \cite[\S 6.3.1,
Theorem~3]{evans-2010-par-dif-equ}}]\label{thm: interior smooth}
    Assume $a^{ij}, b^i, c \in C^\infty(\Omega)$ and $f \in C^\infty(\Omega)$. Suppose
    that $u \in H^1(\Omega)$ is a weak solution of $Lu = f$ in $\Omega$. Then
    $u \in C^\infty(\Omega)$.
\end{theorem}

\paragraph{\leandecl{exists_contDiffOn_of_weakSolution_evans}}
The Lean theorem is stated as
\begin{quote}
\begin{LeanCode}
theorem exists_contDiffOn_of_weakSolution_evans {d : ℕ} {U : Set (EuclideanSpace ℝ (Fin d))}
    (hU : IsOpen U) (_hUb : Bornology.IsBounded U)
    {a : EuclideanSpace ℝ (Fin d) → Fin d → Fin d → ℝ} {b : EuclideanSpace ℝ (Fin d) → Fin d → ℝ}
    {c f u : EuclideanSpace ℝ (Fin d) → ℝ} {G : Fin d → EuclideanSpace ℝ (Fin d) → ℝ}
    (ha : ∀ i j, ContDiffOn ℝ (⊤ : ℕ∞) (fun x => a x i j) U)
    (hb : ∀ i, ContDiffOn ℝ (⊤ : ℕ∞) (fun x => b x i) U)
    (hc : ContDiffOn ℝ (⊤ : ℕ∞) c U) (hf : ContDiffOn ℝ (⊤ : ℕ∞) f U)
    {θ : ℝ} (hθ : 0 < θ)
    (hell : ∀ᵐ x ∂(volume.restrict U), ∀ ξ : Fin d → ℝ,
      θ * ∑ i, ξ i ^ 2 ≤ ∑ i, ∑ j, a x i j * ξ i * ξ j)
    (_hsymm : ∀ x ∈ U, ∀ i j, a x i j = a x j i)
    (hu : MemLp u 2 (volume.restrict U)) (hG : ∀ i, MemLp (G i) 2 (volume.restrict U))
    (hgrad : HasWeakGradOn U u G) (hsol : LocalWeakSol U a b c f u G) :
    ∃ u' : EuclideanSpace ℝ (Fin d) → ℝ,
      ContDiffOn ℝ (⊤ : ℕ∞) u' U ∧ u' =ᵐ[volume.restrict U] u
\end{LeanCode}
\end{quote}

\begin{proof}[Proof sketch]
    Let $B$ be a closed ball contained in $\Omega$. The cutoff in the proof of
    Theorem~\ref{thm: interior h2} gives coefficients on $\mathbb{R}^n$ with bounded
    derivatives of every order that agree with the $a^{ij}$, $b^i$ and $c$ near $B$, and the
    datum $\chi f$ is smooth with compact support. By
    Theorem~\ref{thm: higher interior} at every order the solution has weak
    derivatives of every order in $L^2$ near $B$. To pass from weak derivatives of
    every order to a classical smooth representative, we apply the
    Gagliardo--Nirenberg--Sobolev inequality along a family of functions closed
    under differentiation in steps of $1/(2n)$. Morrey's inequality gives a
    continuous representative of each derivative that a mollification argument
    identifies with a classical derivative
    \cite[\S 6.3.1, Theorem~3]{evans-2010-par-dif-equ}.\footnote{\leandecl{EllipticPdes.Regularity.interior_smooth}
    and \leandecl{EllipticPdes.Regularity.interior_smooth_W12}, based on
    \leandecl{EllipticPdes.Embedding.contDiffOn_of_gradClosed}.} Two continuous
    functions that agree almost everywhere on an open set agree on it, so the
    representatives obtained on the balls agree on their overlaps and define a
    single function of class $C^\infty(\Omega)$ that is equal to $u$ almost
    everywhere.\footnote{\leandecl{EllipticPdes.Regularity.exists_contDiffOn_of_closedBall_ae},
    based on \leandecl{EllipticPdes.Regularity.exists_contDiffOn_of_locally_ae}, and
    the statement for a local weak solution with no boundedness or symmetry
    assumption \leandecl{EllipticPdes.Regularity.exists_contDiffOn_of_localWeakSol}.}
    The formal statement reads $u \in C^\infty(\Omega)$ as the existence of such a
    representative, tests the weak formulation against $C_c^\infty(\Omega)$ as in
    Theorem~\ref{thm: interior h2}, and uses neither the boundedness of $\Omega$ nor
    the symmetry of the $a^{ij}$. The declaration on which it rests differs from the cited statement
    in three respects. Evans assumes $a^{ij}, b^i, c, f \in C^\infty(\Omega)$ on a bounded
    open set, where the declaration takes coefficients with bounded weak derivatives of
    every order and a datum with weak derivatives of every order in $L^2(\Omega)$. Evans
    assumes $u \in H^1(\Omega)$, where the declaration takes $u \in H_0^1(\Omega)$ and the
    cutoff reduction of Theorem~\ref{thm: interior h2} recovers the weaker hypothesis. Evans
    concludes $u \in C^\infty(\Omega)$ for one representative, where the declaration gives
    one on the interior of each $V \Subset \Omega$ and the gluing above assembles them into
    a single representative on $\Omega$. Under the $W^{k,\infty}$ hypotheses the statement asks of
    $a^{ij}$ one thing beyond the bundles, the $W^{1,\infty}$ hypothesis of
    Theorem~\ref{thm: interior h2}, which is read by the base case. Its redundancy
    is the statement that a $W^{1,\infty}$ function has a Lipschitz representative,
    which Mathlib does not prove, so the statement asks for the hypothesis and the
    bundle of order one separately.
\end{proof}

Each of the three statements above is proved in the library in the form that Guo and that Gilbarg
and Trudinger state, from which Theorems~\ref{thm: interior h2}, \ref{thm: higher interior}
and~\ref{thm: interior smooth} follow by the cutoff reduction. For a weak solution $u \in H_0^1(\Omega)$ of \eqref{eqn: dirichlet problem}
with $a^{ij}$ Lipschitz on $\mathbb{R}^n$ and $b^i, c$ bounded, the second weak
derivatives exist in $L^2(V)$ on every compact $V \subseteq \Omega$ and are bounded one
pair of directions at a time, with the constant quantified before $u$ and $f$
\cites[Theorem~VIII.2.2]{guo-2026-par-dif-equ}[Theorem~8.8]{gilbarg-2001-ell-par-dif}.\footnote{\leandecl{EllipticPdes.Regularity.interior_H2_estimate}.}
With $a^{ij} \in W^{m+1,\infty}(\Omega)$ and $b^i, c \in W^{m,\infty}(\Omega)$
with $f \in H^m(\Omega)$, the
solution has weak derivatives to order $m + 2$ on every such $V \Subset \Omega$, each bounded by
$C(\|f\|_{H^m(\Omega)} + \| u \|_{L^2(\Omega)})$
\cite[Theorem~VIII.3.2]{guo-2026-par-dif-equ}.\footnote{\leandecl{EllipticPdes.Regularity.higher_interior_regularity}.}
With those hypotheses at every order the solution has a representative smooth on the
interior of each such $V$
\cites[Theorem~VIII.3.3]{guo-2026-par-dif-equ}[Corollary~8.11]{gilbarg-2001-ell-par-dif}.\footnote{\leandecl{EllipticPdes.Regularity.interior_smooth}.}
In all three statements above the hypothesis on the principal coefficients is $W^{1,\infty}$, which is what
the difference-quotient proof reads and what
\cite[Theorem~8.8]{gilbarg-2001-ell-par-dif} asks.

We thus obtain a
classical solvability statement:

\begin{corollary}[{Classical solvability in the interior,
\cites[\S 6.2.2, Theorem~3]{evans-2010-par-dif-equ}[\S 6.3.1,
Theorem~3]{evans-2010-par-dif-equ}}]\label{cor: classical solvability}
    Let $\Omega \subseteq \mathbb{R}^n$ be bounded and open, $L$ be the operator \eqref{eqn: divergence form operator}, uniformly elliptic in the sense of \eqref{eqn: ellipticity} with $a^{ij} \in C^1(\Omega)$, $a^{ij}, b^i, c \in W^{k,\infty}(\Omega)$, and $f \in H^k(\Omega)$ for every $k \in \mathbb{N}$, and suppose that $b \equiv 0$ and $c \geq 0$ almost everywhere. Then, \eqref{eqn: dirichlet problem} has a unique weak solution $u \in H_0^1(\Omega)$ whose class has, on the interior of each
    compact $V \subseteq \Omega$, a representative that is smooth there and satisfies $Lu = f$
    pointwise almost everywhere there.
\end{corollary}

\paragraph{\leandecl{exists_weakSolution_interior_classical}}
The Lean theorem is stated as
\begin{quote}
\begin{LeanCode}
theorem exists_weakSolution_interior_classical {n : ℕ}
    (Op : FullEllipticOp (n + 1)) {Ω : Set (EuclideanSpace ℝ (Fin (n + 1)))}
    (hΩm : MeasurableSet Ω) (hΩo : IsOpen Ω) (hΩb : Bornology.IsBounded Ω)
    (hb : ∀ i, ∀ᵐ x ∂(volume.restrict Ω), Op.b x i = 0)
    (hc : ∀ᵐ x ∂(volume.restrict Ω), 0 ≤ Op.c x)
    (hA1 : IsC1Coeff Op.toEllipticCoeff)
    (hA : ∀ k : ℕ, IsWkInftyCoeff Op.toEllipticCoeff k)
    (hbc : ∀ k : ℕ, IsWkInftyLower Op k)
    (f : L2D Ω)
    (hf : ∀ k : ℕ, ∃ hfk : HasIteratedWeakDerivOn Ω k f, ∃ M : ℝ, IteratedL2Bound hfk M)
    {V : Set (EuclideanSpace ℝ (Fin (n + 1)))} (hVc : IsCompact V) (hVΩ : V ⊆ Ω) :
    ∃ u : H01 Ω,
      (∀ v : H01 Ω, Op.fullBilin Ω u v = ∫ x in Ω, (f x : ℝ) * ((v : H1amb Ω) 0 x : ℝ)) ∧
      ∃ u' : EuclideanSpace ℝ (Fin (n + 1)) → ℝ,
        u' =ᵐ[volume.restrict (interior V)]
            (extendL2 hΩm ((u : H1amb Ω) 0) : EuclideanSpace ℝ (Fin (n + 1)) → ℝ) ∧
          ContDiffOn ℝ (⊤ : ℕ∞) u' (interior V) ∧
          ∀ᵐ x ∂volume, x ∈ interior V →
            -(∑ i, ∑ j, partialD j (fun y => Op.a y i j * partialD i u' y) x)
              + ∑ i, Op.b x i * partialD i u' x + Op.c x * u' x = f x
\end{LeanCode}
\end{quote}

\begin{proof}[Proof sketch]
    Theorem~\ref{thm: main} provides the unique weak solution and
    Theorem~\ref{thm: interior smooth} provides its smooth representatives on each compactly contained subset of $\Omega$. The fundamental
    lemma of the calculus of variations shows that the weak formulation implies $Lu = f$ pointwise on the interior. The pointwise equation differentiates $a^{ij} D_i u$ classically, and the formal statement therefore also takes $a^{ij} \in C^1(\Omega)$. The hypotheses $b \equiv 0$ and
    $c \geq 0$ make the associated bilinear form coercive with $\gamma = 0$ in
    Theorem~\ref{thm: garding}; these hypotheses enter the proof only through
    Theorem~\ref{thm: main} and are not used by
    Theorem~\ref{thm: interior smooth}.\footnote{\leandecl{EllipticPdes.Regularity.exists_weakSolution_interior_classical}
    which combines \leandecl{EllipticPdes.Sobolev.FullEllipticOp.weak_solution_L2_of_nonneg_zeroth_of_bounded}
    with \leandecl{EllipticPdes.Regularity.interior_smooth} and the pointwise step.
    The smoothness statement alone is \leandecl{EllipticPdes.Regularity.exists_weakSolution_interior_smooth}.}
\end{proof}
\section{Structure of the Lean library}\label{sec: library}

In this section we describe the Lean library in which the proofs and statements of the results in Section~\ref{sec: statements} are formalised. The library consists of three layers: Sobolev spaces, the Poincar\'e reduction, and the Sobolev embedding. 
The solvability results of Section~\ref{sec: statements} are proved in the first two layers and the regularity results are proved in the third. We also include a module graph in which every
result is placed above the results it uses. 

\subsection{Foundational layers}\label{subsec: sobolev
layer}\label{subsec: sobolev embedding layer}

\subsubsection*{Weak-derivative Sobolev layer}

The Sobolev spaces are realised as weak-derivative Hilbert spaces so that a
member is an $L^2$ function together with its $L^2$ gradients; the
ambient space is $L^2(\Omega) \times (L^2(\Omega))^n$, encoded as
$\mathtt{PiLp}\;2$ over $\mathrm{Fin}\,(n+1)$, with coordinate $0$ the
function and coordinate $i{+}1$ the $i$th weak partial derivative. The
weak-gradient relation is expressed as orthogonality to a fixed family of
constraint vectors, so that $W^{1,2}(\Omega)$ is an orthogonal complement and
is therefore closed, complete and a real Hilbert space. The space
$H_0^1(\Omega)$ is the topological closure of the test-function graphs and is
contained in $W^{1,2}(\Omega)$ by an integration by parts without boundary
term.\footnote{\leandecl{EllipticPdes.Sobolev.W12},
\leandecl{EllipticPdes.Sobolev.H01} and
\leandecl{EllipticPdes.Sobolev.testGraph_mem_W12}.}

On this layer the bilinear form $B[u,v] = \sum_{i = 1}^n \langle \partial_i u,
\partial_i v\rangle$ of the Laplacian is bounded and is coercive by the
Poincar\'e inequality obtained by density. The Lax--Milgram theorem in
Mathlib then yields a unique weak solution of the Dirichlet problem for the Poisson equation for every $f \in H^{-1}(\Omega)$.
Data $f \in L^2(\Omega)$ enters through the functional
$v \mapsto \int_\Omega f\,v_0$, which is a contraction
$L^2(\Omega) \to H^{-1}(\Omega)$.\footnote{\leandecl{EllipticPdes.l2Functional}.}
The quantitative part of the Lax--Milgram theorem gives the a priori bound of
Theorem~\ref{thm: lax milgram}.\footnote{\leandecl{EllipticPdes.lax_milgram},
\leandecl{EllipticPdes.poisson_weak_solution} and
\leandecl{EllipticPdes.norm_weak_solution_le}.} This form is the special case
$a^{ij} = \delta^{ij}$, $b = 0$, $c = 0$ in \eqref{eqn: bilinear form}. The remainder of the
library develops the theory of the full operator on this layer along the lines
of the classical divergence-form theory as in \cite{evans-2010-par-dif-equ}. The coefficient bundle consists of measurable
representatives on $\mathbb{R}^n$ that satisfy the ellipticity and entrywise
bounds almost everywhere.\footnote{\leandecl{EllipticPdes.Sobolev.EllipticCoeff}
and \leandecl{EllipticPdes.Sobolev.FullEllipticOp}.} The formalised bilinear form is
bounded with continuity constant $n^2\Lambda$.

The Poincar\'e constant is derived from the geometry of the domain. We transfer
the one-dimensional slice bound along the measure-preserving identification
with $\mathtt{EuclideanSpace}$ and thereby obtain unconditional coercivity of
the form of the Laplacian on every open coordinate box together with the
Poincar\'e inequality on a box.\footnote{\leandecl{laplaceBilin_coercive_euclBox}
and \leandecl{poincare_H01_euclBox}, in the module \texttt{Poincare/BoxSlice}.}
Since a test function on a bounded $\Omega$ and its integrals over a box can be
transferred to a box containing $\Omega$, combining these results yields the
inequality on a bounded
set.\footnote{\leandecl{EllipticPdes.Poincare.poincare_H01_of_bounded}.} Because
the compactness of the embedding is itself proved for the weak-derivative space
by means of the Fr\'echet--Kolmogorov criterion and the estimate of $L^2$
translations by the gradient, the Fredholm alternative and the spectral
decomposition hold on every bounded set with no further
hypothesis.\footnote{\leandecl{embL2_isCompact}.}

\subsubsection*{Poincar\'e reduction layer}

Coercivity of \eqref{eqn: bilinear form} follows from
Theorem~\ref{thm: poincare}, whose proof is sketched in
Section~\ref{sec: statements}. The analytic content of that proof lies in the
following two one-dimensional lemmas, from which the estimate is built up through
intermediate statements that are formalised separately and may be used
independently.

\begin{lemma}[{One-dimensional Cauchy--Schwarz inequality,
\cite[App.~B.2]{evans-2010-par-dif-equ}}]\label{lem: cauchy schwarz}
    Let $f, g$ be continuous on $[a, b]$. Then
    \begin{equation*}
        \Bigl( \int_a^b f(t)\, g(t) \, \dd t \Bigr)^2
          \leq \Bigl( \int_a^b f(t)^2 \, \dd t \Bigr)
               \Bigl( \int_a^b g(t)^2 \, \dd t \Bigr) .
    \end{equation*}
\end{lemma}

\begin{lemma}[One-dimensional Poincar\'e inequality]\label{lem: poincare 1d}
    Let $u$ be continuously differentiable on $[a, b]$ with $u(a) = 0$. Then
    \begin{equation*}
        \int_a^b u(t)^2 \, \dd t \leq \frac{(b - a)^2}{2} \int_a^b u'(t)^2 \, \dd t .
    \end{equation*}
\end{lemma}

Lemma~\ref{lem: cauchy schwarz} is the case $p = q = 2$ of H\"older's
inequality, which the formal proof obtains by the discriminant argument. Lemma~\ref{lem: poincare 1d} then follows from the representation
$u(x) = \int_a^x u'$ by the fundamental theorem of calculus combined with
Lemma~\ref{lem: cauchy schwarz}.\footnote{\leandecl{EllipticPdes.Poincare.intervalIntegral_mul_sq_le}
and \leandecl{EllipticPdes.Poincare.poincare_oneDim}.} We state the
Cauchy--Schwarz inequality as a separate lemma because the continuity bound on
the bilinear form also uses it. Both lemmas are also contained, in a form
prepared for submission upstream, in the first-named author's Mathlib fork alongside
the Gagliardo--Nirenberg--Sobolev inequality. Whereas the sharp Friedrichs
constant would require eigenvalue theory, the one-dimensional argument yields
the explicit constant $(b - a)^2/2$ that the reduction to boxes preserves.

The formalisation devotes one module to each step of the proof of
Theorem~\ref{thm: poincare}, from Lemma~\ref{lem: cauchy schwarz} to the
inequality on a bounded set. The reduction is carried out on a box because Fubini's theorem
and the product measure are already available for boxes in Mathlib. This choice
entails no loss of generality, since enclosing a bounded set in a box transfers
the constant to that set.

\subsubsection*{Sobolev embedding layer}

The Sobolev embedding layer, consisting of three levels, provides the passage from equivalence classes of functions (differing on sets of measure zero) to functions defined everywhere pointwise.

The lowest level consists of a single step. The Gagliardo--Nirenberg--Sobolev
inequality of Mathlib is stated for compactly
supported functions on the whole space. We extend it to weak-derivative classes
on a ball contained in a larger ball and thereby obtain the local form of
Theorem~\ref{thm: sobolev} that is used in the cutoff argument, in which the
cutoff is identically one on the smaller ball.\footnote{The module
\texttt{Embedding/GagliardoNirenberg}.} Morrey's inequality is extended in the
same way and yields a H\"older continuous representative whenever the exponent
exceeds the dimension.\footnote{The modules \texttt{Embedding/Morrey},
\texttt{Embedding/MorreyOneDim} and \texttt{Embedding/RayIntegral}.} 
The radii in the bootstrap decrease only within the
interval between its two given radii. Since the derivative of a representative
is the representative of the derivative, the differentiability argument is
carried out on the inner ball itself.

The intermediate level is the bootstrap over a family closed under weak
differentiation, whose design and step condition are described in the proof of
Theorem~\ref{thm: sobolev}. The bootstrap with the full step $1/n$ applies when
a depth function bounds the supply of derivatives and that with the half step
$1/(2n)$ when derivatives of every order are available.\footnote{The modules \texttt{Embedding/SobolevLadderFullStep} and
\texttt{Embedding/SobolevLadder}.} The exponent-lowering form of the step is the
formal counterpart of the passage through $W^{k-\ell,s}$ for $s \in [1,n)$, a
step that the classical proof treats separately and in which it uses the
boundedness of
$\Omega$.\footnote{\leandecl{EllipticPdes.Embedding.exists_eLpNorm_sobolevConj_le_of_le}.} 

At the highest level weak derivatives are converted into classical ones. A weak
gradient is unique almost everywhere and, if it has a continuous representative,
this representative is a classical Fr\'echet derivative, as we prove by
mollification and uniform convergence on a compact
enlargement.\footnote{The modules \texttt{Embedding/WeakGradUnique} and
\texttt{Embedding/ClassicalDeriv}.} Combining these two results, we show that
the members of a closed family are smooth, as the statement of Theorem~\ref{thm: interior smooth} requires, or, at finite depth,
H\"older regular, as Theorem~\ref{thm: sobolev} requires.\footnote{The modules \texttt{Embedding/SmoothOfGradClosed},
\texttt{Embedding/HolderOfGradClosed} and
\texttt{Regularity/InteriorHolderFinite}.} The closed family is assembled from
the interior estimates.\footnote{The module \texttt{Regularity/IteratedFamily}.}
No separate uniqueness argument is required there, since a single iterated
weak-derivative hypothesis already assigns a class to every index list and the
successor is therefore the closure.

The library also contains a shorter argument that provides an alternative to
this layer in low dimension. Since $H^2$ data alone provides only two derivatives,
the argument consists of a single step that improves the integrability of the
gradient from $L^2$ to $L^6$, after which Morrey's inequality applies for $n \leq 3$.\footnote{The module
\texttt{Embedding/InteriorHolder}.}

\subsection{Module architecture}\label{subsec: architecture}

The library is organised so that the logical dependencies between results are
reflected in the module graph and every import points strictly upward.
Consequently, each result can be reused without the results above it.
Figure~\ref{fig: modules} shows the modules and the imports between them, with
the Sobolev and Poincar\'e layers in its upper part and the modules of solvability
and regularity built on them in its lower part. It is extracted from the import headers of the
modules.

The Poincar\'e reduction is formalised with one module for each step of the
proof of Theorem~\ref{thm: poincare}: the interval Cauchy--Schwarz and one-dimensional
inequalities of Lemmas~\ref{lem: cauchy schwarz} and~\ref{lem: poincare 1d}
(\texttt{Poincare/OneDim}), the per-direction bound on a box (\texttt{Fubini}),
the average over the coordinate directions (\texttt{Domain}), the extension from
test functions to $H_0^1(\Omega)$ (\texttt{Density}), the bridge from a slice
bound to the graph-coordinate hypothesis (\texttt{Geometry}), the transport to
Euclidean space and the statement on a box (\texttt{BoxSlice}), and the enclosure
of a bounded set in a box (\texttt{BoundedDomain}). These modules depend on the
weak-derivative Hilbert spaces and the action of the coefficients
(\texttt{Sobolev/Basic}, \texttt{Sobolev/Coefficients}). The existence and
spectral theory, including its unconditional forms on bounded domains, is
developed in the modules above them.

The regularity theory begins with difference quotients along with their control by the
translation estimate, and by a bound uniform in the step
(\texttt{Regularity/DifferenceQuotient}, \texttt{DiffQuotientBound}). These
estimates are localised by the Caccioppoli inequality
(\texttt{Caccioppoli}). The test functions required in the interior argument
are constructed from the cutoff tower and the calculus of $C^1$ coefficients
(\texttt{CutoffTower}, \texttt{CoeffC1}). This part of the library ends with the
interior $H^2$ estimate of Theorem~\ref{thm: interior h2}
(\texttt{Regularity/Interior}). On the basis of this estimate the Sobolev
bootstrap improves the integrability of the gradient to an exponent at which
Morrey's inequality applies and yields interior H\"older continuity with
exponent $1/2$ (\texttt{Embedding/GagliardoNirenberg}, \texttt{Morrey},
\texttt{InteriorHolder}). Since the bootstrap uses the ray-integral estimates on
which the proof of Morrey's inequality is based, the bootstrap is placed above
Morrey's inequality in the import graph.

The regularity theory is extended beyond the $H^2$ estimate to every order by
induction. In the differentiated weak formulation every derivative of the
coefficients is taken from a $W^{k,\infty}$ bundle
(\texttt{Regularity/DifferentiatedWkInfty}, \texttt{LeibnizWkInfty},
\texttt{CoeffWkInfty}). The induction, for which the cutoff calculus provides
the datum at each step (\texttt{CutoffDeriv}, \texttt{CutoffDatum},
\texttt{DatumPiece}), concludes with Theorem~\ref{thm: higher interior}
(\texttt{HigherInterior}). The output of this theorem, a family of weak
derivatives of every order up to $m + 2$, is assembled in
\texttt{Regularity/IteratedFamily} into the closed family to which the
embedding layer of Section~\ref{subsec: sobolev embedding layer} applies. This
closed family is used in two modules, namely \texttt{Regularity/InteriorSmooth}
at every order and \texttt{Regularity/InteriorHolderFinite} at finite depth.

Two directories are not used in the proofs of the results above. Campanato's
characterisation of H\"older continuity by the decay of the mean oscillation is
proved in both directions (\texttt{Campanato}). No module outside the directory
uses it, since in the library H\"older continuity is derived from Morrey's
inequality. Half-ball geometry, tangential difference quotients, the
admissibility of a cutoff tested against it and the division of a $C^1$ weight are steps of the boundary $H^2$ estimate in \cite[\S 6.3.2]{evans-2010-par-dif-equ} and are proved
(\texttt{Regularity/Boundary}). No boundary regularity estimate has yet been derived from
these steps.

The library distinguishes an elementary core suitable for inclusion in Mathlib,
namely the Cauchy--Schwarz inequality on an interval and the one-dimensional
Poincar\'e inequality added to the first-named author's fork of Mathlib, from the
remaining development built on existing Mathlib results. The division between
reused and new infrastructure was determined mechanically by a query over the
combined declaration graph of both Lean environments. The query locates each
premise together with its home module and therefore allows the division to be
rechecked against the libraries as they evolve.

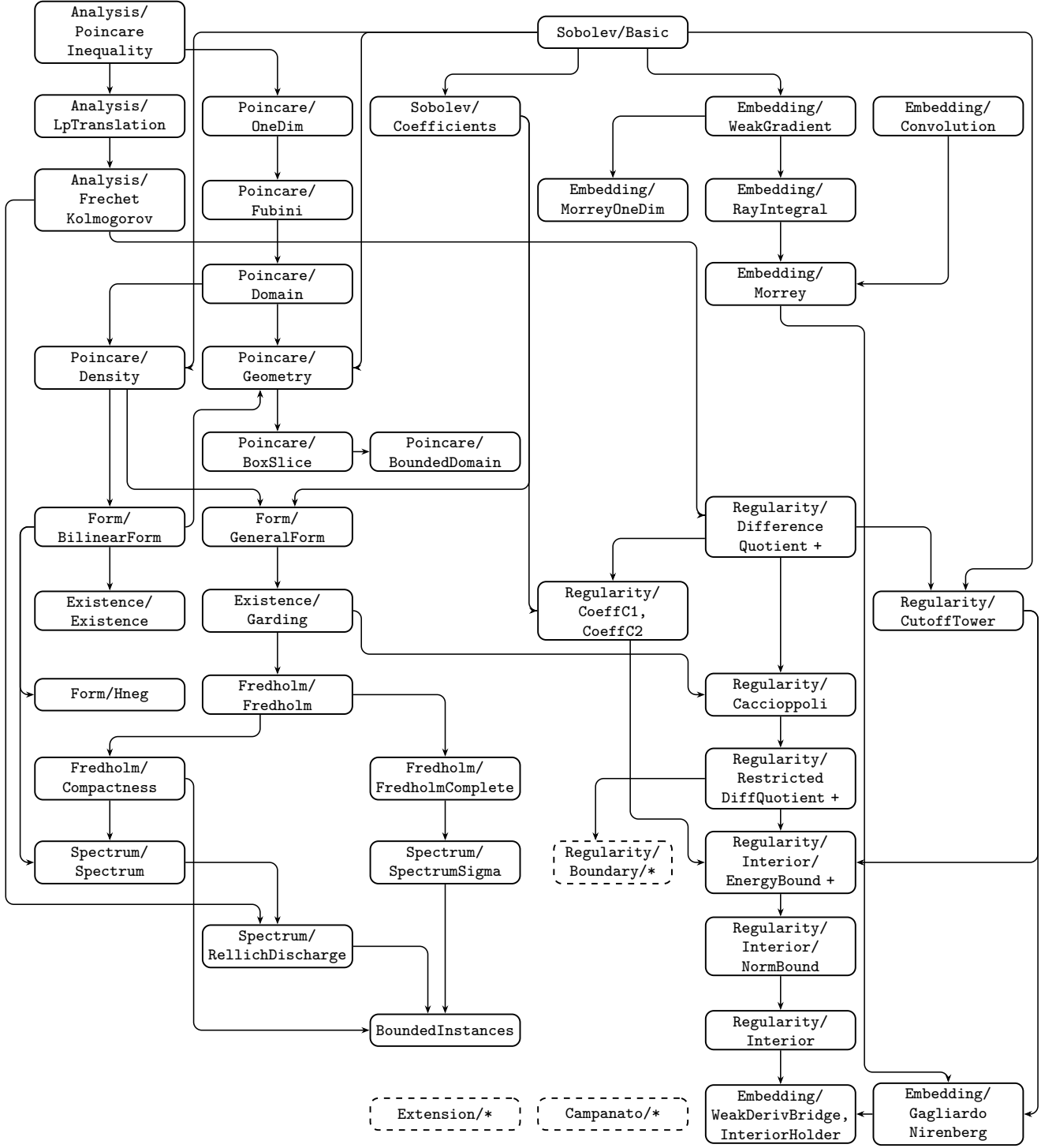
\begin{figure}[p]
  \centering
  \scalebox{0.97}{\begin{tikzpicture}[
      box/.style={draw, rounded corners, font=\scriptsize, text width=2.45cm,
        align=center, minimum height=0.55cm, inner sep=2pt},
      gbox/.style={box, fill=gray!15},
      dbox/.style={box, dashed},
      arr/.style={-{Stealth[length=4pt]}, semithick},
    ]
    \begin{scope}
    \node[box] (pineq)  at (0,0)
    {\texttt{Analysis/}\\\texttt{Poincare}\\\texttt{Inequality}};
    \node[box] (sobB)   at (8.7,0)      {\texttt{Sobolev/Basic}};

    \node[box] (lptr)   at (0,-1.45)    {\texttt{Analysis/}\\\texttt{LpTranslation}};
    \node[box] (onedim) at (2.9,-1.45)  {\texttt{Poincare/}\\\texttt{OneDim}};
    \node[box] (sobC)   at (5.8,-1.45)  {\texttt{Sobolev/}\\\texttt{Coefficients}};
    \node[box] (wgrad)  at (11.6,-1.45) {\texttt{Embedding/}\\\texttt{WeakGradient}};
    \node[box] (conv)   at (14.5,-1.45) {\texttt{Embedding/}\\\texttt{Convolution}};

    \node[box] (fk)     at (0,-2.9)     {\texttt{Analysis/}\\\texttt{Frechet}\\\texttt{Kolmogorov}};
    \node[box] (fub)    at (2.9,-2.9)   {\texttt{Poincare/}\\\texttt{Fubini}};
    \node[box] (m1d)    at (8.7,-2.9)   {\texttt{Embedding/}\\\texttt{MorreyOneDim}};
    \node[box] (ray)    at (11.6,-2.9)  {\texttt{Embedding/}\\\texttt{RayIntegral}};

    \node[box] (dom)    at (2.9,-4.35)  {\texttt{Poincare/}\\\texttt{Domain}};
    \node[box] (morrey) at (11.6,-4.35) {\texttt{Embedding/}\\\texttt{Morrey}};

    \node[box] (dens)   at (0,-5.8)     {\texttt{Poincare/}\\\texttt{Density}};
    \node[box] (geom)   at (2.9,-5.8)   {\texttt{Poincare/}\\\texttt{Geometry}};

    \node[box] (boxs)   at (2.9,-7.25)  {\texttt{Poincare/}\\\texttt{BoxSlice}};
    \node[box] (bdd)    at (5.8,-7.25)  {\texttt{Poincare/}\\\texttt{BoundedDomain}};

    \draw[arr] (pineq) -- (lptr);
    \draw[arr] (lptr) -- (fk);
    \draw[arr] (onedim) -- (fub);
    \draw[arr] (fub) -- (dom);
    \draw[arr] (dom) -- (geom);
    \draw[arr] (geom) -- (boxs);
    \draw[arr] (boxs) -- (bdd);
    \draw[arr] (wgrad) -- (ray);
    \draw[arr] (ray) -- (morrey);
    \draw[arr, rounded corners=4pt] (conv.south) -- (14.5,-4.35) -- (morrey.east);
    \draw[arr, rounded corners=4pt] (sobB.west) -- (1.45,0) -- (1.45,-5.8) -- (dens.east);
    \end{scope}
    \begin{scope}[yshift=-7.1cm]

    \node[box] (bil)   at (0,-1.45)     {\texttt{Form/}\\\texttt{BilinearForm}};
    \node[box] (gen)   at (2.9,-1.45)   {\texttt{Form/}\\\texttt{GeneralForm}};
    \node[box] (dq)    at (11.6,-1.45)
    {\texttt{Regularity/}\\\texttt{Difference}\\\texttt{Quotient +}};

    \node[box] (exi)   at (0,-2.9)      {\texttt{Existence/}\\\texttt{Existence}};
    \node[box] (gar)   at (2.9,-2.9)    {\texttt{Existence/}\\\texttt{Garding}};
    \node[box] (coe)   at (8.7,-2.9)    {\texttt{Regularity/}\\\texttt{CoeffC1,}\\\texttt{CoeffC2}};
    \node[box] (cut)   at (14.5,-2.9)   {\texttt{Regularity/}\\\texttt{CutoffTower}};

    \node[box] (hneg)  at (0,-4.35)     {\texttt{Form/Hneg}};
    \node[box] (fre)   at (2.9,-4.35)   {\texttt{Fredholm/}\\\texttt{Fredholm}};
    \node[box] (cac)   at (11.6,-4.35)  {\texttt{Regularity/}\\\texttt{Caccioppoli}};

    \node[box] (cmp)   at (0,-5.8)      {\texttt{Fredholm/}\\\texttt{Compactness}};
    \node[box] (fcm)   at (5.8,-5.8)    {\texttt{Fredholm/}\\\texttt{FredholmComplete}};
    \node[box] (rdq)   at (11.6,-5.8)
    {\texttt{Regularity/}\\\texttt{Restricted}\\\texttt{DiffQuotient +}};

    \node[box]  (spe)  at (0,-7.25)     {\texttt{Spectrum/}\\\texttt{Spectrum}};
    \node[box]  (sig)  at (5.8,-7.25)   {\texttt{Spectrum/}\\\texttt{SpectrumSigma}};
    \node[dbox, text width=1.9cm] (bnd) at (8.7,-7.25)
      {\texttt{Regularity/}\\\texttt{Boundary/*}};
    \node[box]  (eb)   at (11.6,-7.25)
    {\texttt{Regularity/}\\\texttt{Interior/}\\\texttt{EnergyBound +}};

    \node[box] (rel)   at (2.9,-8.7)    {\texttt{Spectrum/}\\\texttt{RellichDischarge}};
    \node[box] (nb)    at (11.6,-8.7)
    {\texttt{Regularity/}\\\texttt{Interior/}\\\texttt{NormBound}};

    \node[box]  (bi)   at (5.8,-10.15)  {\texttt{BoundedInstances}};
    \node[box]  (int)  at (11.6,-10.15) {\texttt{Regularity/}\\\texttt{Interior}};

    \node[dbox] (camp) at (8.7,-11.6)   {\texttt{Campanato/*}};
    \node[dbox] (ext)  at (5.8,-11.6)   {\texttt{Extension/*}};
    \node[box]  (ih)   at (11.6,-11.6)
    {\texttt{Embedding/}\\\texttt{WeakDerivBridge,}\\\texttt{InteriorHolder}};
    \node[box]  (gns)  at (14.5,-11.6)
    {\texttt{Embedding/}\\\texttt{Gagliardo}\\\texttt{Nirenberg}};

    \draw[arr] (dens) -- (bil);
    \draw[arr] (bil) -- (exi);
    \draw[arr] (gen) -- (gar);
    \draw[arr] (gar) -- (fre);
    \draw[arr] (cmp) -- (spe);
    \draw[arr] (fcm) -- (sig);
    \draw[arr] (sig) -- (bi);
    \draw[arr, rounded corners=4pt] (bil.west) -- (-1.55,-1.45) -- (-1.55,-4.35) -- (hneg.west);
    \draw[arr, rounded corners=4pt] (bil.west) -- (-1.55,-1.45) -- (-1.55,-7.25) -- (spe.west);
    \draw[arr, rounded corners=4pt] (cmp.east) -- (1.45,-5.8) -- (1.45,-10.15) -- (bi.west);

    \draw[arr] (dq) -- (cac);
    \draw[arr, rounded corners=4pt] (gar.east) -- (4.35,-2.9) -- (4.35,-3.65)
      -- (10.0,-3.65) -- (10.0,-4.35) -- (cac.west);
    \draw[arr] (cac) -- (rdq);
    \draw[arr] (rdq) -- (eb);
    \draw[arr, rounded corners=4pt] ([xshift=3mm]coe.south) -- (9.0,-6.5) -- (10.0,-6.5)
      -- (10.0,-7.25) -- (eb.west);
    \draw[arr, rounded corners=4pt] (cut.east) -- (16.05,-2.9) -- (16.05,-7.25) -- (eb.east);
    \draw[arr, rounded corners=4pt] (cut.east) -- (16.05,-2.9) -- (16.05,-11.6) -- (gns.east);
    \draw[arr] (eb) -- (nb);
    \draw[arr] (nb) -- (int);
    \draw[arr] (int) -- (ih);
    \draw[arr] (gns) -- (ih);
    \end{scope}
    \begin{scope}[arr, rounded corners=4pt]
    \draw ([yshift=-3mm]pineq.east) -- (2.9,-0.3) -- (onedim.north);
    \draw ([xshift=-6mm]sobB.south) -- (8.1,-0.725) -- (5.8,-0.725) -- (sobC.north);
    \draw ([xshift=6mm]sobB.south) -- (9.3,-0.725) -- (11.6,-0.725) -- (wgrad.north);
    \draw (wgrad.west) -- (8.7,-1.45) -- (m1d.north);
    \draw (sobB.west) -- (4.35,0) -- (4.35,-5.8) -- (geom.east);
    \draw (dom.west) -- (0,-4.35) -- (dens.north);
    \draw (sobB.east) -- (15.95,0) -- (15.95,-9.275) -- (14.8,-9.275) -- ([xshift=3mm]cut.north);
    \draw (fk.south) -- (0,-3.625) -- (10.15,-3.625) -- (10.15,-8.35) -- ([yshift=2mm]dq.west);
    \draw (fk.west) -- (-1.8,-2.9) -- (-1.8,-15.075) -- (2.6,-15.075) -- ([xshift=-3mm]rel.north);
    \draw (sobC.east) -- (7.25,-1.45) -- (7.25,-10.0) -- (coe.west);
    \draw (sobC.east) -- (7.25,-1.45) -- (7.25,-7.9) -- (3.2,-7.9) -- ([xshift=3mm]gen.north);
    \draw ([xshift=3mm]dens.south) -- (0.3,-7.9) -- (2.6,-7.9) -- ([xshift=-3mm]gen.north);
    \draw (bil.east) -- (1.45,-8.55) -- (1.45,-6.525) -- (2.6,-6.525) -- ([xshift=-3mm]geom.south);
    \draw ([yshift=-2mm]dq.west) -- (8.7,-8.75) -- (coe.north);
    \draw (dq.east) -- (14.2,-8.55) -- ([xshift=-3mm]cut.north);
    \draw (morrey.south) -- (11.6,-5.075) -- (13.05,-5.075) -- (13.05,-17.975) -- (14.5,-17.975) -- (gns.north);
    \draw ([xshift=-3mm]fre.south) -- (2.6,-12.175) -- (0,-12.175) -- (cmp.north);
    \draw (fre.east) -- (5.8,-11.45) -- (fcm.north);
    \draw (spe.east) -- (2.9,-14.35) -- (rel.north);
    \draw (rel.east) -- (5.5,-15.8) -- ([xshift=-3mm]bi.north);
    \draw (rdq.west) -- (8.4,-12.9) -- ([xshift=-3mm]bnd.north);
    \end{scope}
  \end{tikzpicture}}
  \caption{Modules of the library and the imports between them. Each arrow
  represents an \texttt{import} and module paths are given relative to
  \texttt{EllipticPdes/} with long module names broken across lines. Its upper
  part is the weak-derivative Sobolev layer and the Poincar\'e reduction, which are
  independent of the elliptic operator, and its lower part the modules of the
  solvability and regularity theory built on them. A starred path denotes a whole
  directory and a box marked \texttt{+} includes the continuation modules of the
  module that it names. Dashed boxes denote directories that no module in the tree
  imports. Two \texttt{Analysis} modules of auxiliary $L^2$ lemmas,
  \texttt{LpExtendByZero} and \texttt{Euclidean\allowbreak FunctionalNorm}, are
  omitted, as is \texttt{Analysis/\allowbreak LpTranslation\allowbreak Continuity},
  which belongs to the preparatory results for the extension modules.}
  \label{fig: modules}
\end{figure}
\section{Discussion}\label{sec: discussion}\label{subsec: discharge}

Every result claimed in Section~\ref{sec: statements} is proved in Lean and
checked by the kernel. We note four further features of the development.
First, the Poincar\'e inequality is derived from the one-dimensional estimate
alone. Second, the compactness of the embedding is proved by means of the
Fr\'echet--Kolmogorov criterion, so that the Fredholm alternative, the
description of the spectrum, and the resolvent bound require no assumption on
$\Omega$ beyond boundedness. Third, the declaration on which Theorem~\ref{thm: higher interior} rests
takes $a^{ij} \in W^{m+1,\infty}(\Omega)$ and $b^i, c \in W^{m,\infty}(\Omega)$, which is
weaker than the hypothesis $C^{m+1}(\Omega)$ of the classical statement.
Fourth, the variational construction of the eigenvalues is independent of the
spectral theorem for the solution operator and also applies to a semilinear
equation.

Every declaration is proved with no \texttt{sorry} and depends only on the axioms
\texttt{propext}, \texttt{Classical.\allowbreak choice} and \texttt{Quot.sound}. The
one-dimensional Poincar\'e inequality and the Cauchy--Schwarz inequality on an
interval from which it is derived have been added to a fork of Mathlib in a form
suitable for submission upstream.

We measure the extent of a formalisation by its discharge ratio, defined as the
fraction of claimed results that are established by a named Lean declaration in
place of a warrant or an appeal to routine. By a claimed result we mean a
theorem, lemma, corollary or proposition of Section~\ref{sec: statements},
whether it is stated there as a numbered result or in the prose with its
declaration given in a footnote. Sixteen of these results are numbered there, and
the two lemmas on which Theorem~\ref{thm: poincare} rests are stated in
Section~\ref{sec: library}. Since every claimed result
is discharged by a named Lean declaration and no leaf step among them is
warranted, routine or open, the discharge ratio of the present development
is $1$.

\subsection{Conduct of the formalisation}\label{subsec: conduct}

Within the workflow and the acceptance checks described in
Subsection~\ref{subsec: framework}, the authors made three decisions that the
agents did not reach independently, namely the proof of the
Poincar\'e inequality on a box (Theorem~\ref{thm: poincare}) in place of the
representation formula on a ball, the bootstrap in half-steps of $1/(2n)$ in
place of the full step (Theorem~\ref{thm: interior smooth} and the iterations on which
it depends), and the use of a family closed under weak differentiation in place
of an induction on the order (Theorem~\ref{thm: higher interior} and the
bootstrap that uses it).

The four statuses of Subsection~\ref{subsec: framework} are defined with
reference to the one question that the kernel cannot decide, namely whether a
formal statement is the intended one. A discharged obligation corresponds to a
Lean declaration whose type a human reader has compared with the prose
statement. Whereas a warranted obligation is justified by a located citation and
a routine obligation by shared competence, an open obligation has no
justification at all. The kernel verifies only the first component of a
discharge, since a statement that is vacuous, that quantifies in the wrong order
or that permits a constant to depend on the data still elaborates with only the
three standard axioms.

A defect of this kind arose in the present development and was detected in
review before the manuscript was written. Every estimate in the regularity
theory had quantified the solution and the datum before the constant, so that
the type permitted the constant to depend on both and the inequality imposed no
constraint. We corrected the statements so that the constant is quantified
first, in the order now stated in Section~\ref{sec: statements}, and
specified in each statement the quantities on which the constant depends. Since
a successful build is compatible with either order, the discharge ratio stated
above depends on a comparison, made by a human reader, of each type with its
prose statement.

\subsection{Relation to prior work}\label{subsec: audit}

Every general ingredient used in the assembly is available in Mathlib, namely
the Lax--Milgram theorem in the form of the coercive-form equivalence
\texttt{IsCoercive.\allowbreak continuous\allowbreak Linear\allowbreak Equiv\allowbreak OfBilin}
together with its uniqueness counterpart, the Fr\'echet--Riesz representation
theorem, the theory of product measures and Fubini's theorem, and the spectral
theory of compact operators including the Fredholm alternative. The bootstrap
uses the Gagliardo--Nirenberg--Sobolev inequality for compactly supported maps
in the form
\texttt{MeasureTheory.\allowbreak
eLpNorm\allowbreak\_le\allowbreak\_eLpNorm\allowbreak\_fderiv\allowbreak\_of\allowbreak\_eq}
formalised in Mathlib in \cite{vandoorn-2024-int-wit-int}. The extension of this inequality from compactly supported $C^1$ maps to the
Sobolev classes on a domain defined through weak derivatives is carried out in
the present work.

The Sobolev spaces themselves are also constructed in the present work. The
Sobolev spaces available in Mathlib are the Bessel potential spaces in
\cite{doll-2025-for-sch-fun}, defined on the whole space through the Fourier
transform. They therefore provide no predicate of weak differentiability on a
domain of the kind required by the weak formulation of an equation in divergence
form. The De Giorgi--Nash--Moser development in
\cite{armstrong-2026-for-de-gio} contains such a theory, together with a
structure for uniformly elliptic coefficients and a Poincar\'e inequality on
balls proved by integrating the fundamental theorem of calculus along rays. Our
library was built independently of that development. The boundary between the
two was determined by a query over the combined declaration graph of both Lean
environments that locates each premise together with the module in which it is
declared.

The disclosure in Subsection~\ref{subsec: framework} agrees with that given by \cite{chow-2026-lea-for-ham} for a development of comparable
size. Their library is also the closest point of comparison in a second respect,
since it is a differential-geometric development on Mathlib that contains the
divergence theorem on a compact manifold with boundary, Green's identities, a
surface measure and an outward normal. It therefore provides the closest formal
counterpart to the integration by parts with a boundary term that the boundary
theory of the present work would require. Subsection~\ref{subsec: directions} returns to
this comparison.

\subsection{Further directions}\label{subsec: directions}

The most immediate direction concerns the interior theory itself.
Theorem~\ref{thm: interior h2} and Corollary~\ref{cor: classical solvability}
take $a^{ij} \in C^1(\Omega)$, although the declarations on which they rest
require only a Lipschitz bound on the principal coefficients. Removing this
hypothesis requires the identification of $W^{1,\infty}$ functions with their Lipschitz
representative, which Mathlib does not yet contain.

Beyond the interior theory, the parts of \cite[Chapter~6]{evans-2010-par-dif-equ}
and \cite[Chapters~6 and~8]{gilbarg-2001-ell-par-dif} that remain to be formalised
are the boundary theory, Harnack's inequality and the Schauder theory, which
require differing amounts of further work. Of these, boundary $H^2$ regularity in
the sense of \cite[\S 6.3.2, Theorem~4]{evans-2010-par-dif-equ} is the closest to
completion, since the geometry of the half-ball, the tangential difference
quotients together with the cutoff that ensures their admissibility
(\leandecl{cutoffMulOn_tangDiffQuotG_mem_H01}) and the rule for dividing a $C^1$
weight out of a weak derivative (\leandecl{HasWeakDerivOn.of_mul_contDiff_left})
are already formalised. It remains to prove the energy and norm bounds on the half-ball
and to rewrite the equation in nondivergence form so that the tangential estimate
yields a bound on the normal second derivative. This rewriting has no analogue
in the interior argument and constitutes the principal difficulty of the
boundary case.

The global theory requires an integration by parts with a boundary term against
surface measure on $\partial\Omega$, a tool that the interior theory never uses. No
such result is available in Mathlib. The closest formal counterpart is the
differential-geometric library in
\cite{chow-2026-lea-for-ham}, in which the divergence theorem on a compact
manifold with boundary is proved (\leandecl{stokes_compact_via_pou}). That library also proves Green's first
and second identities in the same form, providing a surface measure and
outward normal on the boundary. The identification of the face sum with
$\int_{\partial M} \langle \nu, X \rangle \, \dd \sigma$ against the induced
surface measure enters as a hypothesis of the theorems that state it
(its \leandecl{boundaryFaceSum_eq_surface_integral_of_chartIdentification} and
\leandecl{green_first_eq_boundary_surface_integral}), so that this
identification remains to be proved there. Since that library is directed at
closed manifolds, its main theorem is stated for a model without boundary and
the results on the boundary are developed separately from it. The boundary
$H^2$ estimate on the flat half-ball requires no such theorem, because the test
functions vanish where the boundary term would appear. Its difficulty therefore
lies in the rewriting in nondivergence form described above. The trace theory
needed for a global estimate, by contrast, requires this identification as the
first result to be established.

For the Schauder theory the library contains only Campanato's characterisation
of H\"older continuity and its converse, both of which are proved but neither of
which is yet used. The library contains the maximum principles of \cite[\S 6.4]{evans-2010-par-dif-equ}, namely Theorems~1 and~2 of \S 6.4.1 for a $C^2$
subsolution and both clauses of Theorem~3 of \S 6.4.2, together with the Hopf
boundary point lemma and the weak maximum principle for a weak
subsolution.\footnote{\leandecl{EllipticPdes.Classical.weak_maximum_principle},
\leandecl{EllipticPdes.Classical.weak_maximum_principle_of_nonneg},
\leandecl{EllipticPdes.Classical.strong_maximum_principle},
\leandecl{EllipticPdes.Classical.hopf_lemma} and
\leandecl{EllipticPdes.Sobolev.weak_maximum_principle}. The last of these is
stated for a weak subsolution in $H^1(\Omega)$ and follows
\cite[Theorem~8.1]{gilbarg-2001-ell-par-dif}.}
Harnack's inequality is the furthest from the present library, since our library
contains no form of Moser iteration.

A separate direction concerns the regularity of the coefficients. For equations
in divergence form with bounded measurable coefficients, the interior De
Giorgi--Nash--Moser theory has already been formalised in
\cite{armstrong-2026-for-de-gio}. At the other end of the scale, the Schauder and
Calder\'on--Zygmund theories for smoother coefficients have as yet no formal
counterpart.

A further direction is the completion of the eigenvalue theory. In the present
library the eigenvalues are defined through Rayleigh quotients
(Theorem~\ref{thm: principal eigenvalue} and the sequence constructed from it in
Subsection~\ref{subsec: variational}). Some results that remain to be proved include
the completeness of the variational family in $L^2(\Omega)$, the divergence of
the eigenvalue sequence, and two clauses concerning the principal eigenvalue
that depend on a maximum principle, that is, the positivity of the first
eigenfunction in $\Omega$ and the simplicity of $\lambda_1$. Neither clause is
proved, because the library states the strong maximum principle only for a $C^2$
subsolution and establishes only the weak maximum principle for a weak
subsolution.

\appendix
\begingroup\sloppy

\section{Supporting definitions}\label{sec: appendix transcription}

\label{subsec: appendix definitions}

We collect here the supporting definitions used in the statements of Section~\ref{sec: statements}, each given in mathematical prose beside its Lean definition.

\paragraph{\leandecl{lpExtendByZero}}
\textbf{Extension by zero} as a linear isometry \leaninline{\detokenize{Lp ℝ p (μ.restrict s) →ₗᵢ[ℝ] Lp ℝ p μ}}. A class on the restricted measure is sent to the \leaninline{\detokenize{Lp}} class of its extension by zero, with the \leaninline{\detokenize{Lp}} norm preserved.

\paragraph{\leandecl{gradCLM}}
The continuous linear functional whose coordinate values are the entries of \leaninline{\detokenize{g}} at \leaninline{\detokenize{y}}. This is the form that \leaninline{\detokenize{HasFDerivAt}} requires, assembled from the form in which a weak gradient is given.

\begin{quote}
\begin{LeanCode}
def gradCLM (g : Fin d → EuclideanSpace ℝ (Fin d) → ℝ) (y : EuclideanSpace ℝ (Fin d)) : EuclideanSpace ℝ (Fin d) →L[ℝ] ℝ := ∑ k, g k y • (EuclideanSpace.proj k : EuclideanSpace ℝ (Fin d) →L[ℝ] ℝ)
\end{LeanCode}
\end{quote}

\paragraph{\leandecl{HasWeakGradOn}}
\leaninline{\detokenize{g}} is the pointwise weak gradient of \leaninline{\detokenize{u}} on \leaninline{\detokenize{B}}: integration by parts holds against every smooth compactly supported test function whose support lies in \leaninline{\detokenize{B}}. This mirrors \leaninline{\detokenize{EllipticPdes.Regularity.HasWeakDerivOn}} component-wise but for pointwise functions $u, g_k : \mathbb{R}^{d} \to \mathbb{R}$ (\leaninline{\detokenize{u, gₖ : EuclideanSpace ℝ (Fin d) → ℝ}}).

\paragraph{\leandecl{HasC1Boundary}}
\textbf{$\Omega$ (\leaninline{\detokenize{Ω}}) has $C^1$ (\leaninline{\detokenize{C¹}}) boundary.} Every boundary point admits a chart, which is the hypothesis of Guo's Theorem III.2.2 and of Evans's §5.4 Theorem 1.

\begin{quote}
\begin{LeanCode}
def HasC1Boundary (Ω : Set (EuclideanSpace ℝ (Fin d))) : Prop := ∀ x ∈ frontier Ω, ∃ c : C1Chart d, c.Fits Ω x
\end{LeanCode}
\end{quote}

\paragraph{\leandecl{laplaceBilin}}
The bilinear form of the Laplacian on $H_0^1(\Omega)$ (\leaninline{\detokenize{H₀¹(Ω)}}) as a bounded (continuous) bilinear form, with operator-norm bound \leaninline{\detokenize{d}}.

\paragraph{\leandecl{embL2}}
The coordinate-\leaninline{\detokenize{0}} embedding $H_0^1(\Omega) \hookrightarrow L^2(\Omega)$ (\leaninline{\detokenize{H₀¹(Ω) ↪ L²(Ω)}}), $U \mapsto U 0$ (\leaninline{\detokenize{U ↦ U 0}}), as a continuous linear map: the \leaninline{\detokenize{PiLp}} projection onto coordinate \leaninline{\detokenize{0}} precomposed with the submodule inclusion.

\begin{quote}
\begin{LeanCode}[commandchars=\\\{\}]
def embL2 (Ω : Set (EuclideanSpace ℝ (Fin d))) : H01 Ω →L[ℝ] L2D Ω := (PiLp.proj (\symbol{"1D55C} := ℝ) 2 (fun _ : Fin (d + 1) => L2D Ω) (0 : Fin (d + 1))).comp (H01 Ω).subtypeL
\end{LeanCode}
\end{quote}

\paragraph{\leandecl{IteratedL2Bound}}
\textbf{Uniform $L^2$ (\leaninline{\detokenize{L²}}) bound on an iterated family.} Every derivative up to order \leaninline{\detokenize{k}} is bounded by \leaninline{\detokenize{C}} in $L^2(V)$ (\leaninline{\detokenize{L²(V)}}). Kept apart from \leaninline{\detokenize{HasIteratedWeakDerivOn}} so that existence and estimate can be proved and used separately, matching the shape of \leaninline{\detokenize{interior_H2_estimate}}, which returns the derivative and its bound as separate conjuncts.

\begin{quote}
\begin{LeanCode}
def IteratedL2Bound {u : L2D V} (hu : HasIteratedWeakDerivOn V k u) (C : ℝ) : Prop := ∀ α : List (Fin d), α.length ≤ k → ‖hu.D α‖ ≤ C
\end{LeanCode}
\end{quote}

\paragraph{\leandecl{extendL2}}
\textbf{Extension by zero}, $L^2(\Omega) \hookrightarrow L^2(\mathbb{R}^{d})$ (\leaninline{\detokenize{L2D Ω →ₗᵢ[ℝ] EucL2 d}}). A class on the restricted measure \leaninline{\detokenize{volume.restrict Ω}} is sent to the whole-space $L^2(\mathbb{R}^d)$ (\leaninline{\detokenize{L²(ℝ^d)}}) class of its extension by zero, with the $L^2$ (\leaninline{\detokenize{L²}}) norm preserved. This extension lets the whole-space difference quotients act on gradient data given on a restricted domain \cite[\S 6.3.1]{evans-2010-par-dif-equ}.

\begin{quote}
\begin{LeanCode}
def extendL2 {Ω : Set (EuclideanSpace ℝ (Fin d))} (hΩm : MeasurableSet Ω) : L2D Ω →ₗᵢ[ℝ] EucL2 d := lpExtendByZero volume 2 Ω hΩm
\end{LeanCode}
\end{quote}

\paragraph{\leandecl{iteratedNorm}}
\textbf{\leaninline{\detokenize{H^k(V)}} norm of an iterated family}, summed over lists of directions of length at most \leaninline{\detokenize{k}}.

\begin{quote}
\begin{LeanCode}
def iteratedNorm (H : HasIteratedWeakDerivOn V k u) : ℝ := Real.sqrt (∑ m ∈ Finset.range (k + 1), ∑ α : Fin m → Fin d, ‖H.D (List.ofFn α)‖ ^ 2)
\end{LeanCode}
\end{quote}

\paragraph{\leandecl{LocalWeakSol}}
\textbf{Local weak solution on \leaninline{\detokenize{W}}, on plain function representatives.} \leaninline{\detokenize{u}} with gradient \leaninline{\detokenize{G}}, tested against smooth functions compactly supported in \leaninline{\detokenize{W}}. This is the plain-integral form in which Evans §6.3.1 states the equation, as opposed to the \leaninline{\detokenize{H01}}-and-\leaninline{\detokenize{fullBilin}} formulation that \leaninline{\detokenize{interior_smooth}} takes; \leaninline{\detokenize{isLocalWeakSolution_iff_localWeakSol}} connects it to the predicate \leaninline{\detokenize{IsLocalWeakSolution}} on the ambient space.

\paragraph{\leandecl{partialD}}
The \leaninline{\detokenize{i}}-th classical partial derivative of $\varphi$ (\leaninline{\detokenize{φ}}) (a directional \leaninline{\detokenize{fderiv}}).

\begin{quote}
\begin{LeanCode}
def partialD (i : Fin d) (φ : EuclideanSpace ℝ (Fin d) → ℝ) : EuclideanSpace ℝ (Fin d) → ℝ := fun x => (fderiv ℝ φ x) (EuclideanSpace.single i 1)
\end{LeanCode}
\end{quote}

\paragraph{\leandecl{IsTestFn}}
A smooth, compactly supported test function whose support lies inside $\Omega$ (\leaninline{\detokenize{Ω}}).

\begin{quote}
\begin{LeanCode}
def IsTestFn (Ω : Set (EuclideanSpace ℝ (Fin d))) (φ : EuclideanSpace ℝ (Fin d) → ℝ) : Prop := ContDiff ℝ (⊤ : ℕ∞) φ ∧ HasCompactSupport φ ∧ tsupport φ ⊆ Ω
\end{LeanCode}
\end{quote}

\paragraph{\leandecl{testGraphSet}}
The set of test-function graphs over $\Omega$ (\leaninline{\detokenize{Ω}}).

\begin{quote}
\begin{LeanCode}
def testGraphSet (Ω : Set (EuclideanSpace ℝ (Fin d))) : Set (H1amb Ω) := { U | ∃ (φ : EuclideanSpace ℝ (Fin d) → ℝ) (h : IsTestFn Ω φ), U = h.testGraph }
\end{LeanCode}
\end{quote}

\paragraph{\leandecl{H01}}
$H_0^1(\Omega) = W_0^{1,2}(\Omega)$ (\leaninline{\detokenize{H₀¹(Ω) = W₀^{1,2}(Ω)}}): the closure of the smooth compactly supported functions inside the ambient $H^1$ (\leaninline{\detokenize{H¹}}) space. As a topological closure it is automatically a closed, complete, real Hilbert space.

\begin{quote}
\begin{LeanCode}
def H01 (Ω : Set (EuclideanSpace ℝ (Fin d))) : Submodule ℝ (H1amb Ω) := (Submodule.span ℝ (testGraphSet Ω)).topologicalClosure
\end{LeanCode}
\end{quote}

\paragraph{\leandecl{solOp}}
The \textbf{solution operator} on $L^2(\Omega)$ (\leaninline{\detokenize{L²(Ω)}}) of a coercive form \leaninline{\detokenize{B}}: $G = \iota \circ (B^{\sharp})^{-1} \circ \iota^{\dagger}$ (\leaninline{\detokenize{G = ι ∘ (B♯)⁻¹ ∘ ι†}}), with \leaninline{\detokenize{ι = embL2 Ω}} the Rellich embedding and $(B^{\sharp})^{-1}$ (\leaninline{\detokenize{(B♯)⁻¹}}) the Lax--Milgram inverse of \leaninline{\detokenize{B}}.

\begin{quote}
\begin{LeanCode}
def solOp (B : H01 Ω →L[ℝ] H01 Ω →L[ℝ] ℝ) (hco : IsCoercive B) : L2D Ω →L[ℝ] L2D Ω := (embL2 Ω).comp ((hco.continuousLinearEquivOfBilin.symm : H01 Ω →L[ℝ] H01 Ω).comp (embL2 Ω).adjoint)
\end{LeanCode}
\end{quote}

\paragraph{\leandecl{rayleighValues}}
The values a bilinear form takes on the unit $L^2$ (\leaninline{\detokenize{L²}}) sphere.

\begin{quote}
\begin{LeanCode}
def rayleighValues (B : H01 Ω →L[ℝ] H01 Ω →L[ℝ] ℝ) : Set ℝ := (fun U => B U U) '' rayleighSphere Ω
\end{LeanCode}
\end{quote}

\paragraph{\leandecl{principalEigenvalue}}
\textbf{Principal eigenvalue} of a symmetric coercive form on $H_0^1(\Omega)$ (\leaninline{\detokenize{H₀¹(Ω)}}): the infimum of the Rayleigh quotient \leaninline{\detokenize{B[U, U]}} over the functions of unit $L^2$ (\leaninline{\detokenize{L²}}) norm.

\begin{quote}
\begin{LeanCode}
def principalEigenvalue (B : H01 Ω →L[ℝ] H01 Ω →L[ℝ] ℝ) : ℝ := sInf (rayleighValues B)
\end{LeanCode}
\end{quote}
\endgroup

\setstretch{0.5}
\printbibliography

\hrule

\Addresses{%
  \kmsaddress{AJSF}{Hessah al-Mubarak, Kuwait City, Capital Governorate, State of
  Kuwait}{asoto12@alumni.jh.edu}
  \kmsaddress{KMS}{Department of Mathematics, Johns Hopkins University, Baltimore, MD
  21218, USA}{kmarsh34@jh.edu}
}

\end{document}